\documentclass[12pt]{article}
\usepackage{amsmath,amssymb,bbm,theorem}
\newcommand{\tmmathbf}[1]{\ensuremath{\boldsymbol{#1}}}
\newcommand{\tmop}[1]{\ensuremath{\operatorname{#1}}}
\newtheorem{theorem}{Theorem}[section]
\newtheorem{corollary}[theorem]{Corollary}
\newtheorem{lemma}[theorem]{Lemma}

\newtheorem{definition}[theorem]{Definition}

\numberwithin{equation}{section}
\newenvironment{proof}{\noindent\textbf{Proof\ }}{\hspace*{\fill}$\Box$\medskip}

\begin{document}

\title{The sovability of norm, bilinear and quadratic equations over finite
fields via spectra of graphs}\author{Le Anh Vinh\\
Mathematics Department\\
Harvard University\\
Cambridge, MA 02138, USA\\
vinh@math.harvard.edu}\maketitle

\begin{abstract}
  In this paper we will give a unified proof of several results on the sovability of systems of certain equations over finite fields, which were recently obtained by Fourier analytic methods.
  Roughly speaking, we show that almost all systems of norm,
  bilinear or quadratic equations over finite fields are solvable in any large
  subset of vector spaces over finite fields.
\end{abstract}

\section{Introduction}

The main purpose of this paper is to give a unified proof of several results
on the solvability of systems of certain equations over finite fields, which were recently obtained by Fourier analytic methods.
We will see that after appropriate graph theoretic results are developed, many old and new results
immediately follows. In this section, we discuss the motivation and background results for our work.

Let $\mathbbm{F}_q$ denote a finite field with $q$ elements, where $q$, a power of an odd prime, is viewed as an asymptotic parameter. For $\mathcal{E} \subset \mathbbm{F}^d_q$ ($d \geq 2$), the finite
analogue of the classical Erd\H{o}s distance problem is to determine the smallest
possible cardinality of the set
\[ \Delta (\mathcal{E}) =\{\|\tmmathbf{x}-\tmmathbf{y}\|= (x_1 - y_1)^2 +
   \ldots + (x_d - y_d)^2 : \tmmathbf{x}, \tmmathbf{y} \in \mathcal{E}\}
   \subset \mathbbm{F}_q . \]
The first non-trivial result on the Erd\H{o}s distance problem in vector spaces over
finite fields is due to Bourgain, Katz, and Tao (\cite{bkt}), who showed that if $q$ is a prime, $q \equiv 3$ (mod $4$), then for
every $\varepsilon > 0$ and $\mathcal{E} \subset \mathbbm{F}^2_q$ with
$|\mathcal{E}| \leq C_{\varepsilon} q^2$, there exists $\delta > 0$ such that
$| \Delta (\mathcal{E}) | \geq C_{\delta} q^{\frac{1}{2} + \delta}$ for some
constants $C_{\varepsilon}, C_{\delta}$. The relationship between
$\varepsilon$ and $\delta$ in their arguments, however, is difficult to
determine. In addition, it is quite subtle to go up to higher dimensional cases with these arguments. Iosevich and Rudnev (\cite{ir}) used
Fourier analytic methods to show that there are absolute constants $c_1, c_2 >
0$ such that for any odd prime power $q$ and any set $\mathcal{E} \subset
\mathbbm{F}^d_q$ of cardinality $|\mathcal{E}| \geq c_1 q^{d / 2}$, we have
\[ | \Delta (\mathcal{E}) | \geq c \min \left\{ q, q^{\frac{d - 1}{2}}
   |\mathcal{E}| \right\} . \]
In \cite{vinh-erdos}, the author gave another proof of this result using the graph theoretic method (see also \cite{van} for a similar proof). The (common) main step of these proofs is to estimate the number of occurrences of a fixed distance. It was shown that for
a fixed distance, given that the point set is large, the number of occurrences of
any fixed distance is close to the expected number. This implies that there are many distinct distances occur in a large
point set. In the case of real number field, most of the known results, however, are actually proved in a
stronger form. In order to show that there are at least $g(n)$ distinct distances determined
by an $n$-point set in the plane, one usually proves that for any $n$-point set $P$, there exists
a point $p \in P$ that determines at least $g(n)$ distinct distances to $P$. Chapman et al. (\cite{chapman}) obtained an analogous result in the finite
field setting. They also proved a similar result for the pinned dot
product sets $\Pi_{\tmmathbf{y}} (\mathcal{E}) =\{\tmmathbf{x} \cdot
\tmmathbf{y}: \tmmathbf{x} \in \mathcal{E}\}$. In this paper, we will derive
these results using spectral graph methods.

A classical result due to Furstenberg, Katznelson and Weiss (\cite{fkw}) states that
if $\mathcal{E} \subset \mathbbm{R}^2$ of positive upper Lebesgue density,
then for any $\delta > 0$, the $\delta$-neighborhood of $\mathcal{E}$ contains
a congruent copy of a sufficiently large dilate of every three-point
configuration. An example of Bourgain (\cite{b1}) showed that it is not possible to
replace the thickened set $\mathcal{E}_{\delta}$ by $\mathcal{E}$ for
arbitrary three-point configurations. In the case of $k$-simplex, that is the
$k + 1$ points spanning a $k$-dimensional subspace, Bourgain (\cite{b1}), using
Fourier analytic techniques, showed that a set $\mathcal{E}$ of positive
upper Lebesgue density always contains a sufficiently large dilate of every
non-degenerate $k$-point configuration where $k < d$. In the case $k = d$, the
problem still remains open. Using Fourier analytic method, Akos Magyar (\cite{magyar1,magyar2}) considered this problem over the integer lattice $\mathbbm{Z}^d$. He showed
that a set of positive density will contain a congruent copy of every large
dilate of a non-degenerate $k$-simplex where $d > 2 k + 4$.

Hart and Iosevich (\cite{hi}) made the first investigation in an analog of this
question in finite field geometries. Let $P_k$ denote a $k$-simplex. Given
another $k$-simplex $P^{'}_k$, we say $P_k \sim P_k^{'}$ if there exist $\tau \in
\mathbbm{F}_q^d$, and $O \in SO_d (\mathbbm{F}_q)$, the set of $d$-by-$d$
orthogonal matrices over $\mathbbm{F}_q$, such that $P^{'}_k = O(P_k) + \tau$.
Under this equivalent relation, Hart and Iosevich (\cite{hi}) observed that one may
specify a simplex by the distances determined by its vertices. They showed
that if $\mathcal{E} \subset \mathbbm{F}_q^d$ ($d \geq \binom{k + 1}{2}$) of
cardinality $|\mathcal{E}| \gtrsim C q^{\frac{k d}{k + 1} + \frac{k}{2}}$ then
$\mathcal{E}$ contains a congruent copy of every $k$-simplices (with the exception
of simplices with zero distances). Using graph theoretic method, the author
(\cite{vinh-pg}) showed that the same result holds for $d \geq 2 k$ and $|\mathcal{E}| \gg
q^{\frac{d - 1}{2} + k}$. Here, and throughout, $X \lesssim Y$ means that there exists $C >0$ such that $X \leq C Y$, and $X \ll Y$ means that $X = o(Y)$. Note that serious difficulties arise when the size of simplex is sufficiently large with respect to the ambient dimension. Even
in the case of triangles, the result in \cite{vinh-pg} is only non-trivial for $d \geq
4$. Covert, Hart, Iosevich, and Uriarte-Tuero (\cite{covert}) addressed the case of
triangles in plane over finite fields. They showed that if $\mathcal{E}$ has
density $\geq \rho$ for some $C q^{- 1 / 2} \leq \rho \leq 1$ with a
sufficiently large constant $C > 0$, then the set of triangles determined by
$\mathcal{E}$, upto congruence, has density $\geq c \rho$. In \cite{vinh-tfs}, the author
studied the remaining case: triangles in three-dimensional vector spaces over
finite fields. Using a combination of graph theory method and Fourier
analytic techniques, the author showed that if $\mathcal{E} \subset \mathbbm{F}_q^d$ ($d
\geq 3$) of cardinality $|\mathcal{E}| \gtrsim C q^{\frac{d + 2}{2}}$,
the set of triangles, up to congruence, has density greater than $c$. Using
Fourier analytic techniques, Chapman et al (\cite{chapman}) extended this result to higher dimensional cases. More precisely, they showed that if $|\mathcal{E}|
\gtrsim q^{\frac{d + k}{2}}$ ($d \geq k$) then the set of $k$-simplices, up to
congruence, has density greater than $c$. They also obtained a stronger result
when $\mathcal{E}$ is a subset of the $d$-dimensional unit sphere $S^d
=\{\tmmathbf{x} \in \mathbbm{F}_q^d : \|\tmmathbf{x}\|= 1\}$. In particular,
it was proven (\cite[Theorem 2.15]{chapman}) that if $\mathcal{E} \subset S^d$ of cardinality $|\mathcal{E}| \gtrsim q^{\frac{d + k - 1}{2}}$ then
$\mathcal{E}$ contains a congruent copy of a positive proportion of all
$k$-simplices. In this paper, we will obtain similar results in a more general
setting. Let $Q$ be a non-degenerate quadratic form on $\mathbbm{F}_q^d$. The
$Q$-distance between two points $\tmmathbf{x}, \tmmathbf{y} \in
\mathbbm{F}_q^d$ is defined by $Q (\tmmathbf{x}-\tmmathbf{y}$). We consider
the system $\mathcal{L}$ of $\binom{k}{2}$ equations
\begin{equation}\label{q-i-s} Q (\tmmathbf{x}_i -\tmmathbf{x}_j) = \lambda_{i j}, \tmmathbf{x}_i \in
   \mathcal{E}, i = 1, \ldots, k \end{equation}
over $\mathbbm{F}_q^d$, with variables from arbitrary set $\mathcal{E} \subset
\mathbbm{F}_q^d$. We show that if $|\mathcal{E}| \gg q^{\frac{d - 1}{2} + k -
1}$ then the system (\ref{q-i-s}) is solvable for all $\lambda_{i j} \in \mathbbm{F}_q^{*}$,
and if $|\mathcal{E}| \gg q^{(d + k) / 2}$ then that system \ref{q-i-s} is solvable for at
least $(1 - o (1)) q^{\binom{k}{2}}$ possible choices of $\lambda_{i j} \in
\mathbbm{F}_q$. 

A related question that has recently received attention is the
following. Let $\mathcal{A} \subset \mathbbm{F}_q$, how large does
$\mathcal{A}$ need to be to ensure that $\mathbbm{F}_q^{\ast} \subset
\mathcal{A} \cdot \mathcal{A}+ \ldots +\mathcal{A} \cdot \mathcal{A}$ ($d$
times). Bourgain (\cite{b2}) showed that if $\mathcal{A} \subset \mathbbm{F}_q$ of
cardinality $|\mathcal{A}| \geq C q^{3 / 4}$ then $\mathcal{A} \cdot
\mathcal{A}+\mathcal{A} \cdot \mathcal{A}+\mathcal{A} \cdot
\mathcal{A}=\mathbbm{F}_q$. Glibichuk and Konyagin (\cite{gk}) proved in the case of
prime fields $\mathbbm{Z}_p$ that for $d = 8$, one can take $|\mathcal{A}| >
\sqrt{q}$. Glibichuk (\cite{g}) then extended this result to arbitrary finite
fields. Note that this question can be stated in a more general setting. Let
$\mathcal{E} \subset \mathbbm{F}_q^d$, how large does $\mathcal{E}$ need to
ensure that the equation
\[ \tmmathbf{x} \cdot \tmmathbf{y}= \lambda, \tmmathbf{x}, \tmmathbf{y} \in
   \mathcal{E} \]
is solvable for any $\lambda \in \mathbbm{F}_q^{\ast}$. Hart and Iosevich
(\cite{hi}), using exponential sums, showed that one can take $|\mathcal{E}| > q^{(d +
1) / 2}$ for any $d \geq 2$. In this paper, we will give another proof of this
result using spectral graph methods.

In analogy with the study of simplices in vector spaces over finite fields,
the author (\cite{vinh-beq}) studied the sovability of systems of bilinear equations over
finite fields. More precisely, for any non-degenerate bilinear form $B (\cdot,
\cdot)$ in $\mathbbm{F}_q^d$, we consider the following system of $l \leq
\binom{k}{2}$ equations
\begin{equation}\label{b-i-s} B (\tmmathbf{a}_i, \tmmathbf{a}_j) = \lambda_{i j}, a_i \in \mathcal{E}, i
   = 1, \ldots, k \end{equation}
over $\mathbbm{F}_q^d$, with variables from an arbitrary set $\mathcal{E} \subset
\mathbbm{F}_q^d$. Using character sum machinery and methods from graph theory,
the author (\cite{vinh-beq}) showed that if each variable in the system (\ref{b-i-s}) appears in at
most $t \leq k - 1$ equations and $|\mathcal{E}| \gg q^{\frac{d - 1}{2} +
t}$, then for any $\lambda_{i j} \in \mathbbm{F}_q^{\ast}$, the system (\ref{b-i-s}) has
$(1 + o (1)) q^{- l} |\mathcal{E}|^k$ solutions. Again, serious
difficulties arise when the number of equations that each variable involves is
sufficiently large with respect to the ambient dimension. In particular, that
result is only non-trivial in the range of $d \geq 2 t$. In the case of three
variables and three equations, the author also proved (\cite[Theorem 1.4]{vinh-beq}) that the
system (\ref{b-i-s}) is solvable for $(1 - o (1)) q^3$ triples $(\lambda_{12},
\lambda_{23}, \lambda_{31}) \in (\mathbbm{F}_q^{\ast})^3$ if $|\mathcal{E}|
\gg q^{\frac{d + 2}{2}}$. In this paper, we will extend this result to systems
with many variables. More precisely, we will show that if $\mathcal{E} \subset
\mathbbm{F}_q^d$ of cardinality $|\mathcal{E}| \gg q^{\frac{d + k}{2}}$ then
the system (\ref{b-i-s}) of all $\binom{k}{2}$ equations is solvable for $(1 - o (1))
q^{\binom{k}{2}}$ possible choices of $\lambda_{i j} \in \mathbbm{F}_q$, $1
\leq i < j \leq k$.

We remark here that one can also obtain this result using Fourier analytic methods (for
example, using \cite[Theorem 2.14]{chapman} instead of \cite[Theorem 2.12]{chapman} in the proof of \cite[Theorem 2.13]{chapman}). However, techniques involved in difference problems are
considerable in Fourier analytic proofs. The main advantage of our approach is that we
can obtain all the aforementioned results at once, after computing the
eigenvalues of appropriate graphs. We will also demonstrate our method by
some related results on norm equations and sum-product equations over finite
fields.

\section{Statement of results}\label{resutls-sesb}

\subsection{Subgraphs in $(n, d, \lambda)$-graphs}\label{graph-result-sesb}

For a graph $G$, let $\lambda_1 \geq \lambda_2 \geq \ldots \geq \lambda_k$ be
the eigenvalues of its adjacency matrix. The quantity $\lambda (G) = \max
\{\lambda_2, - \lambda_n \}$ is called the second eigenvalue of $G$. A graph $G
= (V, E)$ is called an $(n, d, \lambda)$-graph if it is $d$-regular, has $n$
vertices, and the second eigenvalue of $G$ is at most $\lambda$. It is well
known (see \cite[Chapter 9]{as} for more details) that if $\lambda$ is much smaller than the
degree $d$, then $G$ has certain random-like properties. Noga Alon (\cite[Theorem 4.10]{ks}) proved
that every large subset of the set of vertices of $(n, d, \lambda)$-graphs
contains the ``correct'' number of copies of any fixed sparse graph.

\begin{theorem} (\cite[Theorem 4.10]{ks}) \label{a-sesb}
  Let $H$ be a fixed graph with $r$ edges, $s$ vertices, and maximum degree
  $\Delta$, and let $G = (V, E)$ be an $(n, d, \lambda)$-graph where $d \leq
  0.9 n$. Let $m < n$ satisfy $m \gg \lambda (n / d)^{\Delta}$. Then, for
  every subset $V' \subset V$ of cardinality $m$, the number of (not
  necessarily induced) copies of $H$ in $V'$ is
  \[ (1 + o (1)) \frac{m^s}{| \tmop{Aut} (H) |} \left( \frac{d}{n} \right)^r .
  \]
\end{theorem}

If we are only interested in the existence of one copy of $H$ then one can
sometimes improve the conditions on $d$ and $\lambda$ in Theorem \ref{a-sesb}. The first
result of this paper is an improvement of the conditions on $d$ and $\lambda$
in Theorem \ref{a-sesb} for complete bipartite graphs. Let $G \times G$ be the bipartite
graph with two identical vertex parts $V (G)$ and $V (G)$. Two vertices $u$
and $v$ in two different parts are connected by an edge if and only if they
are connected by an edge in $G$. For any two subsets $U_1, U_2 \subset V (G)$,
let $G [U_1, U_2]$ be the induced bipartite subgraph of $G \times G$ on $U_1
\times U_2$.

\begin{theorem}\label{bipartite-sesb}
  For any $t \geq s$ and $t \geq 2$, let $G = (V, E)$ be an $(n, d,
  \lambda)$-graph. For every subsets $U_1, U_2 \subset V$ with
  \[ |U_1 | |U_2 | \geq \lambda^2 (n / d)^{t + s}, \]
  the induced subgraph $G [U_1, U_2]$ contains
  \[ (1 + o (1)) \frac{|U_1 |^s |U_2 |^t}{s!t!} \left( \frac{d}{n} \right)^{s
     t} \]
  copies of $K_{s, t}$.
\end{theorem}

Note that the bound in Theorem \ref{bipartite-sesb} is stronger than that in Theorem \ref{a-sesb} when $t >
s$. For small bipartite subgraphs, $K_{2, t}$, we can further improve the bound in Theorem \ref{bipartite-sesb}.

\begin{theorem}\label{book-sesb}
  For any $t \geq 1$, let $G = (V, E)$ be an $(n, d, \lambda)$-graph. For
  every subsets $U_1, U_2 \subset V$ with
  \[ |U_1 | |U_2 | \geq \lambda^2 (n / d)^{t + 1} \]
  the induced subgraph $G [U_1, U_2]$ contains
  \[ (1 + o (1)) \frac{|U_1 |^s |U_2 |^t}{2! t!} \left( \frac{d}{n} \right)^{s
     t} \]
  copies of $K_{2, t}$.
\end{theorem}

In fact, our results could be stated in multi-color versions, which will be
more convenient in later applications. Suppose that a graph $G$ is
edge-colored by a set of finite colors. We call $G$ an $(n, d, \lambda$)-colored
graph if the subgraph of $G$ on each color is an $(n, d (1 + o (1)),
\lambda)$-graph. The following results are multi-color analogues of Theorems \ref{a-sesb}, \ref{bipartite-sesb} and \ref{book-sesb} for $(n, d, \lambda)$-colored graphs.

\begin{theorem} \label{a-c-sesb}
  Let $H$ be a fixed edge-colored graph with $r$ edges, $s$ vertices, and
  maximum degree $\Delta$, and let $G = (V, E)$ be an $(n, d, \lambda)$-colored
  graph, where $d \leq 0.9 n$. Let $m < n$ satisfy $m \gg \lambda (n /
  d)^{\Delta}$. For every subset $V' \subset V$ of cardinality $m$, the number
  of (not necessarily induced) copies of $H$ in $V'$ is
  \[ (1 + o (1)) \frac{m^s}{| \tmop{Aut} (H) |} \left( \frac{d}{n} \right)^r .
  \]
\end{theorem}

\begin{theorem}\label{bipartite-c-sesb}
  For any $t \geq 2$, let $H$ be a fixed edge-colored complete bipartite graph
  $K_{s, t}$ with $s \leq t$. Let $G = (V, E)$ be an $(n, d, \lambda)$-colored
  graph. For every subset $U_1, U_2 \subset V$ with
  \[ |U_1 | |U_2 | \geq \lambda^2 (n / d)^{t + s}, \]
  the induced subgraph $G [U_1, U_2]$ contains
  \[ (1 + o (1)) \frac{|U_1 |^s |U_2 |^t}{\tmop{Aut} (H)} \left( \frac{d}{n}
     \right)^{s t} \]
  copies of $H$.
\end{theorem}

\begin{theorem}\label{book-c-sesb}
  For any $t \geq 1$, let $H$ be a fixed edge-colored complete bipartite graph
  $K_{2, t}$. Let $G = (V, E)$ be an $(n, d, \lambda)$-colored graph. For
  every subset $U_1, U_2 \subset V$ with
  \[ |U_1 | |U_2 | \geq \lambda^2 (n / d)^{t + 1}, \]
  the induced subgraph $G [U_1, U_2]$ contains
  \[ (1 + o (1)) \frac{|U_1 |^2 |U_2 |^t}{\tmop{Aut} (H)} \left( \frac{d}{n}
     \right)^{2 t} \]
  copies of $H$.
\end{theorem}

The proof of Theorem \ref{a-c-sesb} is similar to that of \cite[Theorem 4.10]{ks}, the proofs
of Theorem \ref{bipartite-c-sesb} and Theorem \ref{book-c-sesb} are similar to the proofs of Theorem \ref{bipartite-sesb} and
Theorem \ref{book-sesb}, respectively. To simplify the notation, we will only present the proofs
of single-color results. Note that going from single-color formulations
(Theorems \ref{a-sesb}, \ref{bipartite-sesb} and \ref{book-sesb}) to multi-color formulations (Theorems \ref{a-c-sesb}, \ref{bipartite-c-sesb} and \ref{book-c-sesb})
is just a matter of inserting different letters in a couple of places.

Although we cannot improve the conditions on $d$ and $\lambda$ in Theorem \ref{a-c-sesb} (or equivalently Theorem \ref{a-sesb}),
we will show that if the number of colors is large, under a weaker condition,
any large induced subgraph of an $(n, d, \lambda)$-color graphs contains
almost all possible colorings of small complete subgraphs.

\begin{theorem}\label{colored-subgraph-sesb}
  For any $t \geq 2$. Let $G = (V, E)$ be an $(n, d, \lambda)$-colored graph,
  and let $m < n$ satisfy $m \gg \lambda (n / d)^{t / 2}$. Suppose that the
  color set $\mathcal{C}$ has cardinality $|\mathcal{C}| = (1 - o (1)) n / d$,
  then for every subset $U \subset V$ with cardinality $m$, the induced
  subgraph of $G$ on $U$ contains at least $(1 - o (1))
  |\mathcal{C}|^{\binom{t}{2}}$ possible colorings of $K_t$.
\end{theorem}

The results above could also be considered as a contribution to the
fast-developing comprehensive study of graph theoretical properties of $(n, d,
\lambda)$-graphs, which has recently attracted lots of attention both in
combinatorics and theoretical computer science. For a recent survey about
these fascinating graphs and their properties, we refer the interested reader
to the paper of Krivelevich and Sudakov (\cite{ks}).

\subsection{Norms in sum sets, pinned norms, and norm equations}\label{norm-result-sesb}

Let $\mathbbm{F}_q$ be a finite field with $q = p^d$ elements. Denoting by
$\bar{\mathbbm{F}}$ an algebraic closure of $\mathbbm{F}_q$, and by
$\mathbbm{F}_{q^n} \subset \bar{\mathbbm{F}}$ the unique extension of the
degree $n$ of $\mathbbm{F}$ for $n \geq 1$. The extension $\mathbbm{F}_{q^n}
/\mathbbm{F}_q$ is a Galois extension, with Galois group $G_n$ canonically
isomorphic to $\mathbbm{Z}/\mathbbm{Z}_n$, the isomorphism being the map
$\mathbbm{Z}/\mathbbm{Z}_n \rightarrow G_n$ defined by $1 \mapsto \sigma$,
where $\sigma$ is the Frobenius automorphism of $\mathbbm{F}_{q^n}$ given by
$\sigma (X) = X^q$. Associated to the extension $\mathbbm{F}_n /\mathbbm{F}$,
the norm map $N = N_{\mathbbm{F}_{q^n} /\mathbbm{F}_q} :
\mathbbm{F}_{q^n}^{\ast} \rightarrow \mathbbm{F}^{\ast}_{q^n}$ is defined by
\[ N (X) = \prod_{i = 0}^{n - 1} \sigma^i (X) = \prod_{i = 0}^{n - 1} X^{q^i}
   = X^{\frac{q^{n - 1}}{q - 1}} . \]
The equation $N(X)= \lambda$, for a fixed $\lambda \in \mathbbm{F}_q$, is important in number theory.
Because the extension $\mathbbm{F}_{q^n} /\mathbbm{F}_q$ is separable, the
equation $N (X) = \lambda$ is always solvable with $X \in \mathbbm{F}_{q^n}$
for any $\lambda \in \mathbbm{F}^{\ast}_q$. We are interested in the
solvability of this equation when $X$ is in a sum set of two large subsets of
$\mathbbm{F}_{q^n}$. More precisely, we have the following result.

\begin{theorem}\label{norm-equation-sesb}
  Let $\lambda \in \mathbbm{F}^{\ast}$ and $\mathcal{A}, \mathcal{B} \subseteq
  \mathbbm{F}_{q^n}$, $n \geq 2$. Suppose that $|\mathcal{A}| |\mathcal{B}|
  \geq q^{n + 2}$. Then the equation $N (X + Y) = \lambda$ is solvable in $X
  \in \mathcal{A}$, $Y \in \mathcal{B}$.
\end{theorem}

For any $\mathcal{A}, \mathcal{B} \subset \mathbbm{F}_{q^n}$, define by $N
(\mathcal{A}+\mathcal{B})$ the norm set of the sum set, i.e.
\[ N (\mathcal{A}+\mathcal{B}) =\{N (X + Y) : X \in \mathcal{A}, Y \in
   \mathcal{B}\}. \]
Theorem \ref{norm-equation-sesb} says that if $|\mathcal{A}| |\mathcal{B}| \geq q^{n + 2}$ then
$\mathbbm{F}_q^{\ast} \subset \mathcal{N}(\mathcal{A}+\mathcal{B})$. We will
show that under a slightly stronger condition, one can always find many
elements $X \in \mathcal{A}$ such that the pinned norm set $N_X
(\mathcal{B})$, which is defined by
\[ N_X (\mathcal{B}) =\{N (X + Y) : Y \in \mathcal{B}\}, \]
contains almost all elements in $\mathbbm{F}_q$.

\begin{theorem}\label{norm-pinned-sesb}
  Let $\mathcal{A}, \mathcal{B} \subseteq \mathbbm{F}_n$, $n \geq 2$. Suppose
  that $|\mathcal{A}| \geq |\mathcal{B}|$ and $|\mathcal{A}| |\mathcal{B}| \gg
  q^{n + 2}$. Then there exists a subset $\mathcal{A}'$ of $\mathcal{A}$ with
  cardinality $|\mathcal{A}' | \gtrsim |\mathcal{A}|$ such that for every $X
  \in \mathcal{A}'$, the equation $N (X + Y) = \lambda$ is solvable in $Y \in
  \mathcal{B}$ for at least $(1 - o (1)) q$ values of $\lambda \in
  \mathbbm{F}^{\ast}$.
\end{theorem}

We also obtain the following results on the solvability of systems of norm
equations over finite fields.

\begin{theorem}\label{norm-fs-sesb}
  Let $\mathcal{A} \subseteq \mathbbm{F}_{q^n}$, $n \geq 2$. Consider the
  systems $\mathcal{L}$ of $l \leq \binom{t}{2}$ norm equations
  \begin{equation}\label{s1-sesb}
    N (X_i + X_j) = \lambda_{i j}, \: X_i \in \mathcal{A}, i = 1, \ldots, t.
  \end{equation}
  Suppose that each variable appears in at most $k \leq t - 1$ equations, and
  $|\mathcal{A}| \gg q^{n / 2 + t - 1}$. Then for any $\lambda_{i j} \in
  \mathbbm{F}_q^{\ast}$, the above system has
  \[ (1 + o (1)) q^{- l} |\mathcal{A}|^t \]
  solutions.
\end{theorem}

\begin{theorem}\label{norm-ps-sesb}
  Let $\mathcal{A} \subseteq \mathbbm{F}_{q^n}, n \geq 2$. Consider the system
  (\ref{s1-sesb}) with $\binom{t}{2}$ equations. Suppose that
  $|\mathcal{A}| \gg q^{(n + t) / 2}$, then that system is solvable for at
  least $(1 - o (1)) q^{\binom{t}{2}}$ choices of $\lambda_{i j} \in
  \mathbbm{F}_q$, $1 \leq i < j \leq t$.
\end{theorem}

\subsection{Dot product set and system of bilinear equations} \label{product-result-sesb}

Let $\mathcal{E}, \mathcal{F} \subset \mathbbm{F}_q^d =\mathbbm{F}_q \times
\ldots \times \mathbbm{F}_q$, $d \geq 2$. For any non-degenerate bilinear form
$B (\cdot, \cdot)$ on $\mathbbm{F}_q^d$, define the \ product set of
$\mathcal{E}$ and $\mathcal{F}$ with respect to $B$ by
\[ B (\mathcal{E}, \mathcal{F}) =\{B (\tmmathbf{x}, \tmmathbf{y}) :
   \tmmathbf{x} \in \mathcal{E}, \tmmathbf{y} \in \mathcal{F}\}. \]
Hart and Iosevich (\cite{hi}), using character sum machinery, proved the following
result.

\begin{theorem}(\cite[Theorem 2.1]{hi})\label{bilinear-equation-sesb}
  Let $\mathcal{E}, \mathcal{F} \subset \mathbbm{F}_q^d$. Suppose that
  $|\mathcal{E}| |\mathcal{F}| \geq q^{d + 1}$, then $\mathbbm{F}_q^{\ast}
  \subseteq B (\mathcal{E}, \mathcal{F})$.
\end{theorem}

Let $\mathcal{E} \subset \mathbbm{F}^d_q$, $d \geq 2$. Define the pinned product set by
\[ B_{\tmmathbf{y}} (\mathcal{E}) =\{B (\tmmathbf{x}, \tmmathbf{y}) :
   \tmmathbf{x} \in \mathcal{E}\}. \]
Chapman et al. (\cite{chapman}) obtained the following result using Fourier analytic methods.

\begin{theorem} (\cite[Theorem 2.4]{chapman}) \label{bilinear-old-pinned-sesb}
  Let $\mathcal{E} \subset \mathbbm{F}^d$ ($d \geq 2$) of cardinality
  $|\mathcal{E}| \geq q^{(d + 1) / 2}$. Then there exists a subset
  $\mathcal{E}' \subset \mathcal{E}$ with cardinality $|\mathcal{E}' | \gtrsim
  |\mathcal{E}|$ such that for every $\tmmathbf{y} \in \mathcal{E}'$, one has
  $|B_{\tmmathbf{y}} (\mathcal{E}) | > q / 2$.
\end{theorem}

Note that \cite[Theorem 2.1]{hi} and \cite[Theorem 2.4]{chapman} are stated only for the dot product, but their proofs go through for
any non-degenerate bilinear form without any essential change. As a corollary
of our results in Section \ref{graph-result-sesb}, we will give graph theoretic proofs of
Theorems \ref{bilinear-equation-sesb} and \ref{bilinear-old-pinned-sesb}. In fact, we will prove the following result instead of Theorem \ref{bilinear-old-pinned-sesb}.

\begin{theorem} \label{bilinear-pinned-sesb}
  Let $\mathcal{E} \subset \mathbbm{F}^d_q$ ($d \geq 2$) of cardinality
  $|\mathcal{E}| \gg q^{(d + 1) / 2}$. Then there exist a subset $\mathcal{E}'
  \subset \mathcal{E}$ with cardinality $|\mathcal{E}' | = (1 - o (1))
  |\mathcal{E}|$ such that for every $\tmmathbf{y} \in \mathcal{E}'$, one has
  $|B_{\tmmathbf{y}} (\mathcal{E}) | = (1 - o (1)) q$.
\end{theorem}

Note that the proof of \cite[Theorem 2.14]{chapman} also implies Theorem \ref{bilinear-pinned-sesb} and vice versa. We, however,
relax the condition on $|\mathcal{E}| \geq q^{(d + 1) / 2}$ to $|\mathcal{E}| \gg q^{(d + 1) / 2}$ to simplify our arguments.

In \cite{vinh-beq}, the author studied the solvability of systems of bilinear equations over
finite fields. Following the proof of \cite[Theorem 4.10]{ks}, the author proved the following result.

\begin{theorem}(\cite{vinh-beq})
  Let $\mathcal{E} \subseteq \mathbbm{F}^d_q$, $d \geq 2$. For any
  non-degenerate bilinear form $B (\cdot, \cdot)$ on $\mathbbm{F}_q^d$.
  Consider the systems $\mathcal{L}$ of $l \leq \binom{t}{2}$ bilinear
  equations
  \begin{equation}\label{b-s-sesb}
    B (\tmmathbf{a}_i, \tmmathbf{a}_j) = \lambda_{i j}, \: \tmmathbf{a}_i \in
    \mathcal{A}, i = 1, \ldots, t.
  \end{equation}
  Suppose that $|\mathcal{E}| \gg q^{\frac{d - 1}{2} + t - 1}$ and each
  variable appears in at most $k \leq t - 1$ equations, then the system (\ref{b-s-sesb}) is
  solvable for any $\lambda_{i j} \in \mathbbm{F}_q^{\ast}$, $1 \leq i < j
  \leq t$.
\end{theorem}

As a simple consequence of Theorem \ref{colored-subgraph-sesb} and the construction of product graph in Section \ref{section:bilinear}, we show that under a weaker condition,
say $|\mathcal{A}| \gg q^{(n + t - 1) / 2}$, the system (\ref{b-s-sesb}) is solvable
for almost all possible choices of parameters $\lambda_{i j} \in
\mathbbm{F}^{\ast}$.

\begin{theorem}\label{bilinear-system-sesb}
  Let $\mathcal{E} \subseteq \mathbbm{F}^d_q$ ($d \geq 2$) of cardinality
  $|\mathcal{E}| \gg q^{(n + t - 1) / 2}$. For any non-degenerate bilinear
  form $B (\cdot, \cdot)$ on $\mathbbm{F}_q^d$, consider the systems
  $\mathcal{L}$ of $\binom{t}{2}$ bilinear equations
  \begin{equation}
    B (\tmmathbf{a}_i, \tmmathbf{a}_j) = \lambda_{i j}, \: \tmmathbf{a}_i \in
    \mathcal{A}, i = 1, \ldots, t.
  \end{equation}
  Then the above system is solvable for $(1 - o (1))
  q^{\binom{t}{2}}$ possible choices of $\lambda_{i j} \in \mathbbm{F}$, $1 \leq i < j
  \leq t$.
\end{theorem}

\subsection{Sum-product equations} \label{sp-result-sesb}

In \cite{sarkozy}, S\'ark\"ozy proved that if $\mathcal{A}, \mathcal{B}, \mathcal{C},
\mathcal{D}$ are ``large'' subsets of $\mathbbm{Z}_p$, more precisely,
$|\mathcal{A}| |\mathcal{B}| |\mathcal{C}| |\mathcal{D}| \gg p^3$, then the
sum-product equation
\begin{equation}\label{sp-r} a + b = c d \end{equation}
can be solved with $a \in \mathcal{A}, b \in \mathcal{B}, c \in \mathcal{C}$
and $d \in \mathcal{D}$. Gyarmati and S\'ark\"ozy \cite{gs} generalized this result on
the solvability of equation (\ref{sp-r}) to finite fields. They also studied the
solvability of other (higher degree) algebraic equations with solutions
restricted to ``large'' subsets of $\mathbbm{F}_q$. Using bounds of
multiplicative character sums, Shparlinski \cite{shparlinski} extended the class of sets which
satisfy this property. Furthermore, Garaev \cite{garaev2} considered the equation (\ref{sp-r}) over
special sets $\mathcal{A}, \mathcal{B}, \mathcal{C}, \mathcal{D}$ to obtain
some results on the sum-product problem in finite fields. More precisely, he
proved the following theorem.

\begin{theorem}(\cite{garaev2}) \label{garaev-theorem}
  For any $\mathcal{A} \subseteq \mathbbm{F}_q$, we have
  \[ |\mathcal{A}|^3 \leq \frac{|\mathcal{A} \cdot \mathcal{A}|
     |\mathcal{A}|^2 |\mathcal{A}+\mathcal{A}|}{q} + \sqrt{q|\mathcal{A} \cdot
     \mathcal{A}| |\mathcal{A}|^2 |\mathcal{A}+\mathcal{A}|}, \]
  which implies that
  \[ |\mathcal{A}+\mathcal{A}| |\mathcal{A} \cdot \mathcal{A}| \gg \min
     \left\{ q|\mathcal{A}|, \frac{|\mathcal{A}|^4}{q} \right\} . \]
\end{theorem}

When one of sum or product sets is small, we have an immediate corollary.

\begin{corollary}\label{cor-garaev}
  Suppose that $\mathcal{A} \subset \mathbbm{F}_q$ and $\min
  (|\mathcal{A}+\mathcal{A}|, |\mathcal{A} \cdot \mathcal{A}|) \leq
  C|\mathcal{A}|$ for an absolute constant $C > 0$.
  \begin{enumerate}
    \item If $|\mathcal{A}| \gg q^{2 / 3}$ then $\max
    (|\mathcal{A}+\mathcal{A}|, |\mathcal{A} \cdot \mathcal{A}|) \gg q$.
    
    \item If $|\mathcal{A}| \ll q^{2 / 3}$ then $\max
    (|\mathcal{A}+\mathcal{A}|, |\mathcal{A} \cdot \mathcal{A}|) \gg
    |\mathcal{A}|^3 / q$.
  \end{enumerate}
\end{corollary}

In \cite{vinh-incidence}, the author reproved Theorem \ref{garaev-theorem} using standard tools from spectral graph
theory. Solymosi gave a similar proof in \cite{solymosi}. We will use the same idea to
study the solvability of sum-bilinear equation
\[ a + c = B (\tmmathbf{b}, \tmmathbf{d}), \]
where $a, b \in \mathbbm{F}_q$, $\tmmathbf{b}, \tmmathbf{d} \in
\mathbbm{F}_q^d$, and $B (\cdot, \cdot)$ is any non-degenerate bilinear form.
More precisely, we have the following result.

\begin{theorem} \label{sp-equation-sesb}
  Let $\mathcal{E}, \mathcal{F} \subset \mathbbm{F}_q \times \mathbbm{F}_q^d$
  and $B (\cdot, \cdot)$ be any non-degenerate bilinear form on $\mathbbm{F}_q^d$. Suppose that
  $|\mathcal{E}| |\mathcal{F}| \geq 2 q^{d + 2}$ then the equation
  \[ a + c + \lambda = B (\tmmathbf{b}, \tmmathbf{d}), \: \text{$(a,
     \tmmathbf{b}) \in \mathcal{E}$, $(c, \tmmathbf{d}) \in \mathcal{F}$} \]
  is solvable for any $\lambda \in \mathbbm{F}_q$. 
\end{theorem}

As an easy corollary of Theorem \ref{sp-equation-sesb}, we have the following sum-product estimate,
which can be viewed as an extension of Theorem \ref{garaev-theorem}.

\begin{theorem}\label{spe-sesb}
  For any $\mathcal{A} \subseteq \mathbbm{F}_q$, let $\mathcal{A} \cdot
  \mathcal{A}=\{a a' : a, a' \in \mathcal{A}\}$ and $d\mathcal{A}=\{a_1 +
  \ldots + a_d : a_1, \ldots, a_d \in \mathcal{A}\}$. We have
  \[ |\mathcal{A}|^{2 d - 1} \leq \frac{|\mathcal{A}|^d |\mathcal{A} \cdot
     \mathcal{A}|^{d - 1} |d\mathcal{A}|}{q} + \sqrt{q^d |\mathcal{A}|^d
     |\mathcal{A} \cdot \mathcal{A}|^{d - 1} |d\mathcal{A}|}, \]
  which implies that
  \[ |\mathcal{A} \cdot \mathcal{A}|^{d - 1} |d\mathcal{A}| \gg \min \left(
     q|\mathcal{A}|^{d - 1}, \frac{|\mathcal{A}|^{3 d - 2}}{q^{d - 1}} \right)
     . \]
\end{theorem}

In analogy with the statement of Corollary \ref{cor-garaev} above, we note the following
consequence of Theorem \ref{spe-sesb}.

\begin{corollary}\label{cor-spe-sesb}
  Let $\mathcal{A}$ be an arbitrary subset of $\mathbbm{F}_q$ with cardinality
  $|\mathcal{A}| \gg q^{1 / 2}$.
  \begin{enumerate}
    \item Suppose that $|\mathcal{A} \cdot \mathcal{A}| \leq C|\mathcal{A}|$
    for an absolute constant $C > 0$. If $|\mathcal{A}| \gg q^{d / (2 d - 1)}$
    then $|d\mathcal{A}| \gg q$, and if $|\mathcal{A}| \ll q^{d / (2 d - 1)}$
    then $|d\mathcal{A}| \gg |\mathcal{A}|^{2 d - 1} / q^{d - 1}$.
    
    \item Suppose that $|\mathcal{A}+\mathcal{A}| \leq C|\mathcal{A}|$ for an
    absolute constant $C > 0$ and
    \[ q^{\frac{d + 1}{2 d + 1}} \ll |\mathcal{A}| \ll q^{\frac{d}{2 d - 1}},
    \]
    we have $|\mathcal{A} \cdot \mathcal{A}| \gg |\mathcal{A}| (q /
    |\mathcal{A}|)^{1 / d}$.
  \end{enumerate}
\end{corollary}

Using the machinery developed in this paper, we also can study systems of
sum-product equations over finite fields. More precisely, we have the
following result.

\begin{theorem}\label{sp-system-sesb}
  For any $\mathcal{E} \subseteq \mathbbm{F}_q \times \mathbbm{F}_q^d$ and a
  non-degenerate bilinear form $B (\cdot, \cdot)$ on $\mathbbm{F}_q^d$, 
  consider the system $\mathcal{L}$ of $\binom{t}{2}$ equations
  \begin{equation}\label{sp-s-sesb} a_i + a_j + \lambda_{i j} = B (\tmmathbf{b}_i, \tmmathbf{b}_j), \end{equation}
  with $(a_i, \tmmathbf{b}_i) \in \mathcal{E}$, $1 \leq i \leq t$. Suppose
  that $|\mathcal{E}| \gg q^{\frac{d}{2} + t - 1}$, then the number of
  solutions of the system (\ref{sp-s-sesb}) is close to the expected number
  \[ (1 + o (1)) \frac{|\mathcal{E}|^t}{q^{\binom{t}{2}}} . \]
  In addition, if $|\mathcal{E}| \gg q^{(d + t) / 2}$ then the system (\ref{sp-s-sesb}) is
  solvable for $(1 - o (1)) q^{\binom{t}{2}}$ possible choices of $\lambda_{i
  j} \in \mathbbm{F}_q$, $1 \leq i < j \leq t$.
\end{theorem}

\subsection{Pinned distances and systems of quadratic equations} \label{quadratic-result-sesb}

Let $Q (\cdot)$ be a non-degenerate quadratic form on $\mathbbm{F}_q^d$. Given any
$\tmmathbf{y} \in \mathbbm{F}_q^d$ and $\mathcal{E} \subset \mathbbm{F}_q^d$,
define the pinned distance set by
\[ \Delta_{\tmmathbf{y}}^Q (\mathcal{E}) =\{Q (\tmmathbf{x}-\tmmathbf{y}) :
   \tmmathbf{x} \in \mathcal{E}\}. \]
Chapman et al. \cite{chapman} obtained the following result using
Fourier analysis method.

\begin{theorem}(\cite{chapman}) \label{distance-opinned-sesb}
  Let $\mathcal{E} \subset \mathbbm{F}_q^d, d \geq 2$. Suppose that
  $|\mathcal{E}| \geq q^{\frac{d + 1}{2}}$. There exists a subset
  $\mathcal{E}'$ of $\mathcal{E}$ with $|\mathcal{E}' | \gtrsim |\mathcal{E}|$
  such that for every $y \in \mathcal{E}$, one has
  \[ | \Delta_{\tmmathbf{y}}^Q (\mathcal{E}) | > q / 2, \]
  where $Q(\tmmathbf{x}) = x_1^2+\ldots+x_d^2$.
\end{theorem}

As a corollary of our graph theoretic results, we will present another proof of Theorem \ref{distance-opinned-sesb}. In fact, we will prove a more general result.

\begin{theorem} \label{distance-pinned-sesb}
  Let $Q$ be any non-degenerate quadratic form on $\mathbbm{F}_q^d$. Let
  $\mathcal{E} \subset \mathbbm{F}^d_q$, $d \geq 2$. Suppose that
  $|\mathcal{E}| \gg q^{(d + 1) / 2}$. There exists a subset $\mathcal{E}'
  \subset \mathcal{E}$ with cardinality $|\mathcal{E}' | = (1 - o (1))
  |\mathcal{E}|$ such that for every $\tmmathbf{y} \in \mathcal{E}'$, one has
  $| \Delta^Q_{\tmmathbf{y}} (\mathcal{E}) | = (1 - o (1)) q$.
\end{theorem}

Note that the proof of \cite[Theorem 2.3]{chapman} implies Theorem \ref{distance-pinned-sesb} and vice versa. We again relax
the condition on $|\mathcal{E}| \geq q^{(d + 1) / 2}$ to $|\mathcal{E}| \gg
q^{(d + 1) / 2}$ to simplify the argument in the proof.

Next we will prove the following result on the solvability of system of
quadratic equations (or equivalently, the existence of the simplices over finite fields).

\begin{theorem} \label{simplex-system-sesb}
  For any non-degenerate quadratic form $Q$ on $\mathbbm{F}_q^d$ and any set
  $\mathcal{E} \subset \mathbbm{F}_q^d$, we consider the following system of
  $l \leq \binom{t}{2}$ equations
  \[ Q (\tmmathbf{a}_i -\tmmathbf{a}_j) = \lambda_{i j}, a_i \in \mathcal{E},
     i = 1, \ldots, t. \]
  If $|\mathcal{E}| \gg q^{(d + t - 1) / 2}$ then this system is solvable for
  at least $(1 - o (1)) q^{\binom{t}{2}}$ possible choices of $\lambda_{i j}
  \in \mathbbm{F}_q$.
\end{theorem}

We obtain a stronger result when $\mathcal{E}$ is a subset of the
$d$-dimensional $Q$-sphere $S^d(Q) =\{\tmmathbf{x} \in \mathbbm{F}_q^d :
Q(\tmmathbf{x})= 1\}$.

\begin{theorem} \label{simplex-sphere-sesb}
  For any non-degenerate quadratic form $Q$ and any set
  $\mathcal{E} \subset S^d(Q)$, we consider the following system of $l \leq
  \binom{t}{2}$ equations
  \[ Q (\tmmathbf{a}_i -\tmmathbf{a}_j) = \lambda_{i j}, a_i \in \mathcal{E},
     i = 1, \ldots, t. \]
  If $|\mathcal{E}| \gg q^{(d + t - 2) / 2}$ then this system is solvable for
  at least $(1 - o (1)) q^{\binom{t}{2}}$ possible choices of $\lambda_{i j}
  \in \mathbbm{F}_q$.
\end{theorem}

Note that Theorems \ref{simplex-system-sesb} and \ref{simplex-sphere-sesb} could be obtained in a similar way as in the proofs of \cite[Theorem 2.13]{chapman} and \cite[Theorem 2.16]{chapman}. The proofs given here, however, are different, which are applicable to a large variety of problems. In addition, the results in Section \ref{graph-result-sesb} can also be extended for directed $(n,
d, \lambda)$-graphs. All the proofs go through without any essential changes
by replacing Theorem \ref{edge-t-sesb} and Corollary \ref{edge-c-sesb} by their corresponding directed
versions (\cite[Lemma 3.1]{van}). Applying these results, we can study the
sovability of systems of general equations
\[ P (\tmmathbf{a}_i -\tmmathbf{a}_j) = \lambda_{i j}, \tmmathbf{a}_i \in
   \mathcal{E}, 1 \leq i \leq t \]
in $\mathbbm{F}_q^d$, where $\mathcal{E}$ is a large subset of
$\mathbbm{F}_q^d$ and $P \in \mathbbm{F}_q [x_1, \ldots, x_d]$. One can show
that for a large family of non-degenerate polynomials $P$, if
$|\mathcal{E}| \gg q^{(d + t - 1) / 2}$, that system is solvable for at
least $c q^{\binom{t}{2}}$ possible choices of $\lambda_{i j} \in
\mathbbm{F}_q$. However, in order to keep this paper concise, we will restrict
our discussion only to results on undirected $(n, d, \lambda)$-graphs and
their applications.

\section{Properties of pseudo-random graphs}

In this section, we recall some results on the distribution of edges in $(n, d,
\lambda)$-graphs. For two (not necessarily) disjoint subsets of vertices $U,
W \subset V$, let $e (U, W)$ be the number of ordered pairs $(u, w)$ such that
$u \in U$, $w \in W$, and $(u, w)$ is an edge of $G$. For a vertex $v$ of $G$,
let $N (v)$ denote the set of vertices of $G$ adjacent to $v$ and let $d (v)$
denote its degree. Similarly, for a subset $U$ of the vertex set, let $N_U (v) = N
(v) \cap U$ and $d_U (v) = |N_U (v) |$. We first recall the following two
well-known facts (see, for example, \cite{as}).

\begin{theorem} (\cite[Theorem 9.2.4]{as})\label{edge-t-sesb}
  Let $G = (V, E)$ be an $(n, d, \lambda)$-colored graph. For any subset $U$
  of $V$, we have
  \[ \sum_{v \in V} (d_U (v) - d|U| / n)^2 < \lambda^2 |U|. \]
\end{theorem}

The following result is an easy corollary of Theorem \ref{edge-t-sesb}.

\begin{corollary}(\cite[Corollary 9.2.5]{as})\label{edge-c-sesb}
  Let $G = (V, E)$ be an $(n, d, \lambda)$-graph. For any two sets $B, C
  \subset V$, we have
  \[ \left| e (B, C) - \frac{d|B | |C|}{n} \right| \leq \lambda \sqrt{|B| |C|}. \]
\end{corollary}

We also need a similar result for number of paths of length two in $(n, d,
\lambda)$-graphs.

\begin{lemma}\label{path-sesb}
  Let $G = (V, E)$ be an $(n, d, \lambda)$-graph. For any two subsets $B, C
  \subset V$, let $p_2 (B, C)$ be the number of paths of length two with the
  midpoint in $B$ and two endpoints in $C$ (i.e. the number of triples $(c_1,
  b, c_2)$ such that $b \in B$, $c_1, c_2 \in C$, $(b, c_1), (b, c_2) \in E
  (G)$). We have
  \[ \left| p_2 (B, C) - \left( \frac{d}{n} \right)^2 |B| |C|^2 \right| \leq 2
     \frac{\lambda d}{n} |B|^{1 / 2} |C|^{3 / 2} + \lambda^2 |C|. \]
\end{lemma}

\begin{proof}
  It follows from Theorem \ref{edge-t-sesb} that
  \begin{equation} \sum_{v \in B} (|d_C (v) - \frac{d}{n} |C|)^2 \leq \sum_{v \in V} (|d_C
     (v) - \frac{d}{n} |C|)^2 \leq \lambda^2 |C|, \end{equation}
  which implies that
  \begin{equation}\label{e1-sesb} \left| \sum_{v \in B} (N_C (v))^2 + \left( \frac{d}{n} \right)^2 |B||C|^2
     - \frac{2 d}{n} |C| \sum_{v \in B} N_C (v) \right| \leq \lambda^2 |C|. \end{equation}
  From Corollary \ref{edge-c-sesb}, we have
  \begin{equation}\label{e2-sesb} \left| \sum_{v \in B} N_C^{} (v) - \frac{d}{n} |B| |C| \right| \leq
     \lambda \sqrt{|B| |C|}, \end{equation}
  Putting (\ref{e1-sesb}) and (\ref{e2-sesb}) together, we have
  \[ \left| \sum_{v \in B} (N_C (v))^2 - \left( \frac{d}{n} \right)^2 |B|
     |C|^2 \right| \leq 2 \frac{\lambda d}{n} |B|^{1 / 2} |C|^{3 / 2} +
     \lambda^2 |C|,\]
  completing the proof of the lemma.
\end{proof}

\section{Complete bipartite subgraphs - Proof of Theorem \ref{bipartite-sesb}}

Let $U_1, U_2$ \ be any subsets of $V = V (G)$. For any $y_1, \ldots, y_t \in
U_2$, let
\[ \mathcal{S}_{y_1, \ldots, y_t} (U_1) =\{x \in U_1 : (x, y_i) \in E (G), 1
   \leq i \leq t\}, \]
\[ \tmmathbf{S}_{y_1, \ldots, y_t} (U_1) = |\mathcal{S}_{y_1, \ldots, y_t}
   (U_1) |. \]
When $\tmmathbf{S}_{y_1, \ldots, y_k} (U_1) \geq 1$, we say that the base
$(y_1, \ldots, y_k)$ is extendable to $k$-stars with roots in $U_1$. In order
to make our inductive argument work, we will need the following definition.

\begin{definition}
  Let $f, g, h$ be any three functions on the same variables. We say that
  \[ f = \tilde{o} (g, h) \]
  if $f = o (g)$ when $h = o (g)$, and $f = O (h)$ otherwise.
\end{definition}

Theorem \ref{bipartite-sesb} for the star $K_{1, t}$ follows immediately from the following
estimate.

\begin{lemma}\label{star-sesb}
  Let $G = (V, E)$ be an $(n, d, \lambda)$-graph. For any $t \geq 1$ and two subsets $U_1, U_2 \subset V$, we have
  \[ \sum_{y_1, \ldots, y_t \in U_2} \tmmathbf{S}_{y_1, \ldots, y_t} (U_1) =
     \left( \frac{d}{n} \right)^t |U_1 | |U_2 |^t + \tilde{o} \left( \left(
     \frac{d}{n} \right)^t |U_1 | |U_2 |^t, \frac{n \lambda^2}{d} |U_2 |^{t -
     1} \right) . \]
\end{lemma}

\begin{proof}
  The proof proceeds by induction. The base case $t = 1$ follows immediately
  from Corollary \ref{edge-c-sesb} and the fact that
  \begin{equation}\label{e3-sesb} \lambda \sqrt{|U_1 | |U_2 |} = \tilde{o} \left( \frac{d}{n} |U_1 | |U_2
     |, \frac{n \lambda^2}{d} \right) . \end{equation}
  Suppose that the claim holds for $t - 1 \geq 1$, we show that it also holds
  for $t$. Note that
  \[ \sum_{y_1, \ldots, y_t \in U_2} \tmmathbf{S}_{y_1, \ldots, y_t} (U_1) =
     \sum_{y_1, \ldots, y_{t - 1} \in U_2} e (U_2, \mathcal{S}_{y_1, ., y_{t -
     1}} (U_1)) . \]
  Hence, it follows from Corollary \ref{edge-c-sesb} and (\ref{e3-sesb}) that
  \begin{eqnarray}
    &  & \sum_{y_1, \ldots, y_t \in U_2} \tmmathbf{S}_{y_1, \ldots, y_t}
    (U_1) = \sum_{y_1, \ldots, y_{t - 1} \in U_2} \left( \frac{d|U_2
    |\tmmathbf{S}_{y_1, ., y_{t - 1}} (U_1)}{n} + O \left( \lambda \sqrt{|U_2
    |\tmmathbf{S}_{y_1, ., y_{t - 1}} (U_1)} \right) \right) \nonumber\\
    &  & = \sum_{y_1, \ldots, y_{t - 1} \in U_2} \left( \frac{d|U_2
    |\tmmathbf{S}_{y_1, ., y_{t - 1}} (U_1)}{n} + \tilde{o} \left( \frac{d|U_2
    |\tmmathbf{S}_{y_1, ., y_{t - 1}} (U_1)}{n}, \frac{n \lambda^2}{d} \right)
    \right) . \label{e4-sesb}
  \end{eqnarray}
  By induction hypothesis,
  \begin{equation}\label{e5-sesb} \sum_{y_1, \ldots, y_{t - 1} \in U_2} \tmmathbf{S}_{y_1, ., y_{t - 1}}
     (U_1) = \left( \frac{d}{n} \right)^{t - 1} |U_1 | |U_2 |^{t - 1} +
     \tilde{o} \left( \left( \frac{d}{n} \right)^{t - 1} |U_1 | |U_2 |^{t -
     1}, \frac{n \lambda^2}{d} |U_2 |^{t - 2} \right) . \end{equation}
  The lemma follows immediately from (\ref{e4-sesb}), (\ref{e5-sesb}) and the additivity of function
  $\tilde{o}$.
\end{proof}

We now give a full proof of Theorem \ref{bipartite-sesb}. Let $U_1, U_2$ \ be any subsets of $V
= V (G)$. For any $y_1, \ldots, y_t \in U_2$, let
\[ \mathcal{K}_{y_1, \ldots, y_t}^s (U_1) =\{x_1, \ldots, x_s \in U_1 : (x_j,
   y_i) \in E (G), 1 \leq j \leq s, 1 \leq i \leq t\}, \]
\[ \tmmathbf{K}_{y_1, \ldots, y_t}^s (U) = |\mathcal{K}_{y_1, \ldots, y_t}^s
   (U) |, \]
then it is clear that
\[ \tmmathbf{K}_{y_1, \ldots, y_t}^s (U) = (\tmmathbf{S}_{y_1, \ldots, y_t}
   (U))^s . \]
Similarly, when $\tmmathbf{K}^s_{y_1, \ldots, y_k} (U_1) \geq 1$, we say that
the base $(y_1, \ldots, y_k)$ is extendable to $K_{s, t}$ graphs with the $s$-parts in
$U_1$. Theorem \ref{bipartite-sesb} follows immediately from the following estimate.

\begin{lemma}\label{bipartite-lemma-sesb}
  Let $G = (V, E)$ be an $(n, d, \lambda)$-graph. For any $t \geq s \geq 1$ and two subsets $U_1, U_2 \subset V$, we have
  \[ \sum_{y_1, \ldots, y_t \in U_2} \tmmathbf{K}^s_{y_1, \ldots, y_t} (U_1) =
     \left( \frac{d}{n} \right)^{t s} |U_1 |^s |U_2 |^t + \tilde{o} \left(
     \left( \frac{d}{n} \right)^{t s} |U_1 |^s |U_2 |^t, \left( \frac{n}{d}
     \right)^{t^2} \lambda^{2 t} |U_1 |^{s - t} \right) . \]
\end{lemma}

\begin{proof}
  The proof proceeds by induction on $s$. We first consider the base case $s =
  1$. Lemma \ref{star-sesb} for the stars implies that
  \[ \sum_{y_1, \ldots, y_t \in U_2} \tmmathbf{K}^1_{y_1, \ldots, y_t} (U_1) =
     \left( \frac{d}{n} \right)^t |U_1 |^1 |U_2 |^t + \tilde{o} \left( \left(
     \frac{d}{n} \right)^t |U_1 | |U_2 |^t, \frac{n \lambda^2}{q} |U_2 |^{t -
     1} \right) . \]
  By Young's inequality,
  \[ \frac{n \lambda^2}{q} |U_2 |^{t - 1} \leq \frac{t - 1}{t} \left(
     \frac{d}{n} \right)^t |U_1 | |U_2 |^t + \frac{1}{t} \left( \frac{n}{d}
     \right)^{t^2} \lambda^{2 t} |U_1 |^{1 - t}, \]
  which implies that
  \[ \frac{n \lambda^2}{q} |U_2 |^{t - 1} = \tilde{o} \left( \left(
     \frac{d}{n} \right)^t |U_1 | |U_2 |^t, \left( \frac{n}{d} \right)^{t^2}
     \lambda^{2 t} |U_1 |^{1 - t} \right) . \]
  The base case $s = 1$ follows. Suppose that the claim holds for $s -
  1 \geq 1$, we show that it also holds for $s$. Note that
  \begin{eqnarray}
    \sum_{y_1, \ldots, y_t \in U_2} \tmmathbf{K}^s_{y_1, \ldots, y_t} (U_1) &
    = & \sum_{z_1, \ldots, z_s \in U_1} \tmmathbf{K}_{z_1, \ldots, z_s}^t
    (U_2) \nonumber\\
    & = & \sum_{z_1, \ldots, z_{s - 1} \in U_1} \sum_{y_1, \ldots, y_t \in
    \mathcal{S}_{z_1, \ldots, z_{s - 1}} (U_2)} \tmmathbf{S}_{y_1, \ldots,
    y_t} (U_1), \label{e6-sesb}
  \end{eqnarray}
  because both sides equal the number of ordered (possibly degenerate) $K_{s,
  t}$ in $G [U_1, U_2]$. It follows from Lemma \ref{star-sesb} that
  \begin{eqnarray}
    \sum_{y_1, \ldots, y_t \in \mathcal{S}_{z_1, \ldots, z_{s - 1}} (U_2)}
    \tmmathbf{S}_{y_1, \ldots, y_t} (U_1) & = & \left( \frac{d}{n} \right)^t
    |U_1 | (\tmmathbf{S}_{z_1, \ldots, z_{s - 1}} (U_2))^t \nonumber\\
    &  & + \tilde{o} \left( \left( \frac{d}{n} \right)^t |U_1 |
    (\tmmathbf{S}_{z_1, \ldots, z_{s - 1}} (U_2))^t, \frac{n \lambda^2}{d}
    (\tmmathbf{S}_{z_1, \ldots, z_{s - 1}} (U_2))^{t - 1} \right) . \label{e7-sesb}
  \end{eqnarray}
  By H\"older's inequality,
  \begin{eqnarray}
    &  & \sum_{z_1, \ldots, z_{s - 1} \in U_1} \frac{n \lambda^2}{d}
    (\tmmathbf{S}_{z_1, \ldots, z_{s - 1}} (U_2))^{t - 1} \nonumber\\
    & \leq & \left( \sum_{z_1, \ldots, z_{s - 1} \in U_1} \left( \frac{d}{n}
    \right)^t |U_1 | (\tmmathbf{S}_{z_1, \ldots, z_{s - 1}} (U_2))^t
    \right)^{(t - 1) / t} \left( \sum_{z_1, \ldots, z_{s - 1} \in U_1} q^{t^2}
    \lambda^{2 t} |U_1 |^{- (t - 1)} \right)^{1 / t} \nonumber\\
    & = & \tilde{o} \left( \sum_{z_1, \ldots, z_{s - 1} \in U_1} \left(
    \frac{d}{n} \right)^t |U_1 | (\tmmathbf{S}_{z_1, \ldots, z_{s - 1}}
    (U_2))^t, \left( \frac{n}{d} \right)^{t^2} \lambda^{2 t} |U_1 |^{s - t}
    \right) . \nonumber
  \end{eqnarray}
  Hence,
  \begin{equation}\label{e8-sesb} \sum_{z_1, \ldots, z_{s - 1} \in U_1} \frac{n \lambda^2}{d}
     (\tmmathbf{S}_{z_1, \ldots, z_{s - 1}} (U_2))^{t - 1} = \tilde{o} \left(
     \sum_{z_1, \ldots, z_{s - 1} \in U_1} \left( \frac{d}{n} \right)^t |U_1 |
     (\tmmathbf{S}_{z_1, \ldots, z_{s - 1}} (U_2))^t, q^{t^2} \lambda^{2 t}
     |U_1 |^{s - t} \right) . \end{equation}
  Besides, by induction hypothesis,
  \begin{eqnarray}
    &  & \sum_{z_1, \ldots, z_{s - 1} \in U_1} \left( \frac{d}{n} \right)^t
    |U_1 | (\tmmathbf{S}_{z_1, \ldots, z_{s - 1}} (U_2))^t = \left(
    \frac{d}{n} \right)^t |U_1 | \sum_{z_1, \ldots, z_{s - 1} \in U_1}
    \tmmathbf{K}^t_{z_1, \ldots, z_{s - 1}} (U_2) \nonumber\\
    & = & \left( \frac{d}{n} \right)^t |U_1 | \left\{ \left( \frac{d}{n}
    \right)^{(s - 1) t} |U_1 |^{t - 1} |U_2 |^s + \tilde{o} \left( \left(
    \frac{d}{n} \right)^{(s - 1) t} |U_1 |^{t - 1} |U_2 |^s, \left(
    \frac{n}{d} \right)^{t^2} \lambda^{2 t} |U_1 |^{s - 1 - t} \right)
    \right\} \nonumber\\
    & = & \left( \frac{d}{n} \right)^{s t} |U_1 |^t |U_2 |^s + \tilde{o} 
    \left( \left( \frac{d}{n} \right)^{s t} |U_1 |^t |U_2 |^s, \left(
    \frac{n}{d} \right)^{t (t - 1)} \lambda^{2 t} |U_1 |^{s - t} \right) .
    \label{e9-sesb}
  \end{eqnarray}
  Putting (\ref{e6-sesb}), (\ref{e7-sesb}), (\ref{e8-sesb}) and (\ref{e9-sesb}) together, the lemma follows.
\end{proof}

\section{Proof of Theorem \ref{book-sesb}}

Theorem \ref{book-sesb} follows immediately from the following estimate on the number of
$K_{2,t}$ subgraphs in $(n, d, \lambda)$-graphs.

\begin{lemma}\label{book-lemma-sesb}
  Let $G = (V, E)$ be an $(n, d, \lambda)$-graph. For any $t \geq 1$ and two subsets $U_1, U_2 \subset V$, we have
  \[ \sum_{y_1, \ldots, y_t \in U_2} \tmmathbf{K}^2_{y_1, \ldots, y_t} (U_1) =
     \left( \frac{d}{n} \right)^{2 t} |U_1 |^2 |U_2 |^t + \tilde{o} \left(
     \left( \frac{d}{n} \right)^{2 t} |U_1 |^2 |U_2 |^t, \lambda^4 \left(
     \frac{n}{d} \right)^2 |U_2 |^{t - 2} \right) . \]
\end{lemma}

\begin{proof}
  The proof proceeds by induction. It
  follows from Lemma \ref{path-sesb} that
  \[ \sum_{y_1 \in U_2} \tmmathbf{K}^2_{y_1} (U_1) = \left( \frac{d}{n}
     \right)^2 |U_1 |^2 |U_2 | + \tilde{o} \left( \left( \frac{d}{n} \right)^2
     |U_1 |^2 |U_2 |, \lambda^2 |U_1 | \right) . \]
  By the Cauchy-Schwarz inequality,
  \begin{equation}\label{e10-sesb} \lambda^2 |U_1 | = \tilde{o} \left( \left( \frac{d}{n} \right)^2 |U_1
     |^2 |U_2 |, \left( \frac{n}{d} \right)^2 \lambda^4 |U_2 |^{- 1} \right) .
  \end{equation}
  Hence, the base case $t = 1$ follows. Suppose that the claim holds for
  $t$, we show that it also holds for $t + 1$. Note that
  \begin{equation}\label{e11-sesb} \sum_{y_1, \ldots, y_t \in U_2} \tmmathbf{K}^2_{y_1, \ldots, y_t} (U_1) =
     \sum_{y_1, \ldots, y_{t - 1} \in U_2}  \sum_{x_1, x_2 \in
     \mathcal{S}_{y_1, ., y_{t - 1}} (U_1)} \mathcal{S}_{x_1, x_2} (U_2) . \end{equation}
  It follows from Lemma \ref{path-sesb} and (\ref{e10-sesb}) that
  \begin{eqnarray}
    \sum_{x_1, x_2 \in \mathcal{S}_{y_1, ., y_{t - 1}} (U_1)}
    \mathcal{S}_{x_1, x_2} (U_2) & = & \left( \frac{d}{n} \right)^2 |U_2 |
    (\tmmathbf{S}_{y_1, \ldots, y_{t - 1}} (U_1))^2 + \nonumber\\
    & + & \tilde{o} \left( \left( \frac{d}{n} \right)^2 |U_2 |
    (\tmmathbf{S}_{y_1, \ldots, y_{t - 1} (U_1)})^2, \lambda^4 \left(
    \frac{n}{d} \right)^2 |U_2 |^{- 1} \right) \label{e12-sesb}
  \end{eqnarray}
  By induction hypothesis,
  \begin{eqnarray}
    \sum_{y_1, \ldots, y_{t - 1} \in U_2} (\tmmathbf{S}_{y_1, ., y_{t - 1}}
    (U_1))^2 & = & \left( \frac{d}{n} \right)^{2 (t - 1)} |U_1 |^2 |U_2 |^{2
    (t - 1)} \nonumber\\
    &  & + \tilde{o} \left( \left( \frac{d}{n} \right)^{2 t} |U_1 |^2 |U_2
    |^{2 (t - 1)}, \lambda^4 \left( \frac{n}{d} \right)^2 |U_2 |^{t - 3}
    \right) . \label{e13-sesb}
  \end{eqnarray}
  Putting (\ref{e11-sesb}), (\ref{e12-sesb}) and (\ref{e13-sesb}) together, we complete the proof of Lemma \ref{book-lemma-sesb}.
\end{proof}

\section{Complete edge-colored subgraphs}

Suppose that $G = (V, E)$ is an $(n, d, \lambda)$-colored graph. Let $E^{r_i}
(G)$ be the set of $r_i$-colored edges of $G$. Let $U_1, U_2$ \ be any subsets
of $V = V (G)$. For any $t$ colors $r_1, \ldots, r_t$, and for any $y_1,
\ldots, y_t \in U_2$, define
\[ \mathcal{S}_{y_1, \ldots, y_t}^{r_1, \ldots, r_t} (U_1) =\{x \in U_1 : (x,
   y_i) \in E^{r_i} (G), 1 \leq i \leq t\}, \]
\[ \tmmathbf{S}_{y_1, \ldots, y_t}^{r_1, \ldots, r_t} (U_1) =
   |\mathcal{S}_{y_1, \ldots, y_t}^{r_1, \ldots, r_t} (U_1) |, \]
and
\begin{equation*}
 \mathcal{I}_{y_1, \ldots, y_t}^{r_1, \ldots, r_t} (U_1) = \left\{
        \begin{array}{ll}
        1 & \mbox{if} \, \tmmathbf{S}_{y_1, \ldots, y_t}^{r_1, \ldots, r_t} (U_1) \geq 1\\
        0 & \mbox{otherwise.}
    \end{array}  \right.
\end{equation*}
When $\mathcal{I}_{y_1, \ldots, y_t}^{r_1, \ldots, r_t} (U_1) = 1$, we say
that the base $(y_1, \ldots, y_k)$ is extendable to $k$-stars of type $(r_1,
\ldots, r_t)$ with roots in $U_1$. For any $t - 1$ colors $r_1, \ldots, r_{t -
1}$, the following lemma says that if $U_1, U_2$ are two large subsets of the
vertex set of an $(n, d, \lambda)$-colored graph, almost all $t$-tuples
$(y_1, \ldots, y) \in U_2^t$ are extendable to stars of type $(r_1, \ldots,
r_t)$ with roots in $U_1$. This lemma is not necessary for the proof of Theorem
\ref{colored-subgraph-sesb}, but it could be of independent interest.

\begin{lemma}\label{e-pinned-sesb}
  For any $t \geq 2$ and $t$ colors $r_1, \ldots, r_t$, let $G = (V, E)$ be
  an $(n, d, \lambda)$-colored graph, and let $U_1, U_2 \subset V$. Suppose
  that
  \[ |U_1 | |U_2 | \gg \lambda^2 \left( \frac{n}{d} \right)^{t + 1}, \]
  then
  \[ \sum_{y_1, \ldots, y_t \in U_2} \mathcal{I}_{y_1, \ldots, y_t}^{r_1,
     \ldots, r_t} (U_1) = (1 - o (1)) |U_2 |^t . \]
\end{lemma}

\begin{proof}
  By Cauchy-Schwartz inequality, we have
  \[ \left( \sum_{y_1, \ldots, y_t \in U_2} \tmmathbf{S}_{y_1, \ldots,
     y_t}^{r_1, \ldots, r_t} (U_1) \right)^2 \leq \sum_{y_1, \ldots, y_t \in
     U_2} \mathcal{I}_{y_1, \ldots, y_t}^{r_1, \ldots, r_t} (U_1) \sum_{y_1,
     \ldots, y_t \in U_2} (\tmmathbf{S}_{y_1, \ldots, y_t}^{r_1, \ldots, r_t}
     (U_1))^2 . \]
  It follows from Lemma \ref{star-sesb} that
  \begin{equation}\label{e14-sesb} \sum_{y_1, \ldots, y_t \in U_2} \tmmathbf{S}_{y_1, \ldots, y_t}^{r_1,
     \ldots, r_t} (U_1) = (1 + o (1)) \left( \frac{n}{d} \right)^t |U_1 | |U_2
     |^t, \end{equation}
  and from Lemma \ref{book-sesb} that
  \begin{equation}\label{e15-sesb} \sum_{y_1, \ldots, y_t \in U_2} (\tmmathbf{S}_{y_1, \ldots, y_t}^{r_1,
     \ldots, r_t} (U_1))^2 = \sum_{y_1, \ldots, y_t \in U_2} (1 + o (1))
     \left( \frac{n}{d} \right)^{2 t} |U_1 |^2 |U_2 |^t . \end{equation}
  Putting (\ref{e14-sesb}) and (\ref{e15-sesb}) together, we have
  \[ \sum_{y_1, \ldots, y_t \in U_2} \mathcal{I}_{y_1, \ldots, y_{t -
     1}}^{r_1, \ldots, r_{t - 1}} (U_1) \geq (1 - o (1)) |U_2 |^t . \]
  The upper bound is trivial, and the lemma follows.
\end{proof}

\subsection{Proof of Theorem \ref{colored-subgraph-sesb}}

We are now ready to give a proof of Theorem \ref{colored-subgraph-sesb}. We note that with the graph
theoretic results above, the argument in the proof of \cite[Theorem 2.13]{chapman} works for any $(n, d,
\lambda)$-colored graphs, and we will follow it closely. To make our
inductive argument work we will need the following definition, which is also
taken from \cite{chapman}.

\begin{definition}
  Given $U \subset V$, let $\mathcal{U} \subset U^t \equiv U \times \ldots
  \times U$, $t \geq 2$. Define
  \[ \mathcal{U}_{t - 1} =\{(y_1, \ldots, y_{t - 1}) : (y_1, \ldots, y_{t -
     1}, y_t) \in \mathcal{U}\}. \]
  Moreover, for each $(y_1, \ldots, y_{t - 1}) \in \mathcal{U}_{t - 1}$,
  define
  \[ \mathcal{U}(y_1, \ldots, y_{t - 1}) =\{y_t : (y_1, \ldots, y_{t - 1},
     y_t) \in \mathcal{U}\} \subset U. \]
\end{definition}

For any two sets $U_1 \subset U_2$, we say that $U_1 \sim U_2$ if and only if
$|U_1 | = (1 - o (1)) |U_2 |$. We now can state a slightly stronger version of
Lemma \ref{e-pinned-sesb}.

\begin{lemma}\label{pinned-sesb}
  For any $t \geq 2$ and $t - 1$ colors $r_1, \ldots, r_{t - 1}$, let $G = (V,
  E)$ be an $(n, d, \lambda)$-colored graph. Suppose that the color set
  $\mathcal{C}$ has cardinality $|\mathcal{C}| = (1 - o (1)) n / d$. For any
  subset $U \subset V$ of cardinality $|U| \gg \lambda (n / d)^{t / 2}$ and
  for any $\mathcal{U} \subset U^t = U \times \ldots \times U$ with
  $\mathcal{U} \sim |U|^t$, we have
  \[ \sum_{(y_1, \ldots, y_{t - 1}) \in \mathcal{U}_{t - 1}} \sum_{r_1,
     \ldots, r_{t - 1} \in \mathcal{C}} \mathcal{I}_{y_1, \ldots, y_{t -
     1}}^{r_1, \ldots, r_{t - 1}} (\mathcal{U}(y_1, \ldots, y_{t - 1})) = (1 -
     o (1)) |\mathcal{U}_{t - 1} | |\mathcal{C}|^{t - 1} . \]
\end{lemma}

\begin{proof}
  The upper bound is trivial so it suffices to show that
  \[ \sum_{(y_1, \ldots, y_{t - 1}) \in \mathcal{U}_{t - 1}} \sum_{r_1,
     \ldots, r_{t - 1} \in \mathcal{C}} \mathcal{I}_{y_1, \ldots, y_{t -
     1}}^{r_1, \ldots, r_{t - 1}} (\mathcal{U}(y_1, \ldots, y_{t - 1})) \geq
     (1 - o (1)) |\mathcal{U}_{t - 1} | |\mathcal{C}|^{t - 1} . \]
  Note that we have $(1 - o (1)) n / d$ colors, so$|E (G) | = (1 - o (1)) n^2
  / 2$. Hence,
  \begin{eqnarray}
    &  & \sum_{(y_1, \ldots, y_{t - 1}) \in \mathcal{U}_{t - 1}} \sum_{r_1,
    \ldots, r_{t - 1} \in \mathcal{C}} \tmmathbf{S}_{y_1, \ldots, y_{t -
    1}}^{r_1, \ldots, r_{t - 1}} (\mathcal{U}(y_1, \ldots, y_{t - 1}))
    \nonumber\\
    & \geq & \sum_{y_1, \ldots, y_{t - 1} \in U} \sum_{r_1, \ldots, r_{t - 1}
    \in \mathcal{C}} \mathcal{S}_{y_1, \ldots, y_{t - 1}}^{r_1, \ldots, r_{t -
    1}} (U) - (|U|^t -\mathcal{U}) \nonumber\\
    & = & (1 - o (1)) |U|^t . \label{e16-sesb}
  \end{eqnarray}
  On the other hand, it follows from Lemma \ref{book-lemma-sesb} that
  \begin{eqnarray}
    &  & \sum_{(y_1, \ldots, y_{t - 1}) \in \mathcal{U}_{t - 1}} \sum_{r_1,
    \ldots, r_{t - 1} \in \mathcal{C}} (\tmmathbf{S}_{y_1, \ldots, y_{t -
    1}}^{r_1, \ldots, r_{t - 1}} (\mathcal{U}(y_1, \ldots, y_{t - 1})))^2
    \nonumber\\
    & \leq & \sum_{(y_1, \ldots, y_{t - 1}) \in U^{t - 1}} \sum_{r_1, \ldots,
    r_{t - 1} \in \mathcal{C}} (\tmmathbf{S}_{y_1, \ldots, y_{t - 1}}^{r_1,
    \ldots, r_{t - 1}} (U))^2 = (1 + o (1)) |U|^{t + 1} \left( \frac{d}{n}
    \right)^{2 (t - 1)} |\mathcal{C}|^{t - 1} \nonumber\\
    & = & (1 + o (1)) |U|^{t - 1} / |\mathcal{C}|^{t - 1}, \label{e17-sesb}
  \end{eqnarray}
  since $|\mathcal{C}| = (1 - o (1)) n / d$. The lemma now follows from (\ref{e16-sesb}), (\ref{e17-sesb}),
  and the Cauchy-Schwarz inequality.
\end{proof}

By the pigeon-hole principle, we have an immediate corollary of Lemma \ref{pinned-sesb}.

\begin{corollary} \label{main-cor-sesb}
  For any $t \geq 2$ and $t - 1$ colors $r_1, \ldots, r_{t - 1}$, let $G = (V,
  E)$ be an $(n, d, \lambda)$-colored graph. Suppose that the color set
  $\mathcal{C}$ has cardinality $|\mathcal{C}| = (1 - o (1)) n / d$. For any
  subset $U \subset V$ of cardinality $|U| \gg \lambda (n / d)^{t / 2}$ and
  for any $\mathcal{U} \subset U^t = U \times \ldots \times U$ with
  $\mathcal{U} \sim |U|^t$, there exists a subset $\mathcal{U}^{(t - 1)}
  \subset \mathcal{U}_{t - 1}$ with $\mathcal{U}^{(t - 1)} \sim \mathcal{U}_{t
  - 1}$ such that
  \[ \sum_{r_1, \ldots, r_{t - 1} \in \mathcal{C}} \mathcal{I}_{y_1, \ldots,
     y_{t - 1}}^{r_1, \ldots, r_{t - 1}} (\mathcal{U}(y_1, \ldots, y_{t - 1}))
     = (1 - o (1)) |\mathcal{C}|^{t - 1} \]
  for every $(y_1, \ldots, y_{t - 1}) \in \mathcal{U}^{(t - 1)}$.
\end{corollary}

This corollary says that for any large set $\mathcal{U} \sim U^t$, there
exists a large subset $\mathcal{U}^{(t - 1)} \subset \mathcal{U}_{t - 1}
\subset U^{t - 1}$ such that $\mathcal{U}^{(t - 1)} \sim |U|^{t - 1}$, and any
$(t - 1)$-tuple $(y_1, \ldots, y_{t - 1}) \in \mathcal{U}^{(t - 1)}$ is
extendable to at least $(1 - o (1)) |\mathcal{C}|^{t - 1}$ types of $(t -
1)$-stars with roots in $\mathcal{U}(y_1, \ldots, y_{t - 1})$. The rest of the
proof is easy. For any set $U \subset V$ with cardinality $|U| \gg \lambda (n
/ d)^{t / 2}$, we construct $t$ sets $\mathcal{U}^{(t)}, \ldots,
\mathcal{U}^{(1)}$ inductively as follows. Let $\mathcal{U}^{(t)} = U^t = U
\times \ldots \times U$. Since $|U| \gg \lambda (n / d)^{t / 2}$, from
Corollary \ref{main-cor-sesb}, we can choose $\mathcal{U}^{(t - 1)} \subset
\mathcal{U}^{(t)}_{t - 1}$ such that any $(t - 1)$-tuple $(y_1, \ldots, y_{t -
1})$ in $\mathcal{U}^{(t - 1)}$ is extendable to at least $(1 - o (1))
|\mathcal{C}|^{t - 1}$ types of $(t - 1)$-stars with roots in
$\mathcal{U}^{(t)} (y_1, \ldots, y_{t - 1})$. Suppose that we have constructed
$\mathcal{U}^{(t)}, \ldots, \mathcal{U}^{(i)}$ ($i \geq 2$) \ such that $\mathcal{U}^{(j)} \sim U^j$ for $i \leq j \leq t - 1$, 
and \ any $j$-tuple $(y_1, \ldots, y_j)$ in
$\mathcal{U}^{(j)}$ is extendable to at least $(1 - o (1)) |\mathcal{C}|^j$
types of $j$-stars with roots in $\mathcal{U}^{(j + 1)} (y_1, \ldots, y_j)$.
Since
\[ |U| \gg \lambda (n / d)^{t / 2} \geq \lambda (n / d)^{i / 2}, \]
from Corollary \ref{main-cor-sesb} again, we can choose $\mathcal{U}^{(i - 1)} \subset
\mathcal{U}^{(i)}_{i - 1}$ such that $\mathcal{U}^{(i-1)} \sim U^{i-1}$, and any $(t - 1)$-tuple $(y_1, \ldots, y_{t -
1})$ in $\mathcal{U}^{(i - 1)}$ is extendable to at least $(1 - o (1))
|\mathcal{C}|^{i - 1}$ types of $(i - 1)$-stars with roots in
$\mathcal{U}^{(i)} (y_1, \ldots, y_{i - 1})$. Repeat the process until we have
constructed $\mathcal{U}^{(1)} \subset \mathcal{U}^{(2)}_1 \subset U$ of
cardinality $|\mathcal{U}^{(1)} | \sim |U|$. Take any vertex $v$ in
$\mathcal{U}^{(1)}$, it is clear that we can extend from $v$ to at least $(1 -
o (1)) |\mathcal{C}|^{\binom{t}{2}}$ types of edge-colored $K_t$ in $U^t$. This
completes the proof of Theorem \ref{colored-subgraph-sesb}.

\section{Projective norm graphs}

We first recall the construction of the projective norm graphs
$\mathcal{N}\mathcal{G}_{q, n} (\lambda)$ in \cite{ap}. It is possible to get a
slightly better version of these (\cite{ars}), but this makes no essential difference
for our purpose here. The construction is the following. Let $n \geq 2$ be an
integer and $q$ be an odd prime power. Let $\mathbbm{F}_q$ be the finite field
of $q$ elements, and $\mathbbm{F}_{q^n}$ be the unique extension of degree $n$
of $\mathbbm{F}_q$. The vertex set of the graph
$\mathcal{N}\mathcal{G}_{q, n} (\lambda)$ is the set $V =\mathbbm{F}_{q^n}$.
Two distinct vertices $X$ and $Y \in V$ are adjacent if and only if $N (X + Y)
= \lambda$, where the norm $N$ is defined as in Section \ref{norm-result-sesb} with an extension $N (0) =
0$. For any $\lambda \in \mathbbm{F}_q^{\ast}$, the equation $N (X) = \lambda$
has a solution, and if $X_0$ is a given solution, the set of solutions is in
one-to-one correspondence with solutions of $N (X) = 1$, which by Hilbert's
Theorem 90 (or by direct proof) is given by $X = \sigma (Y) Y^{- 1} =
Y^{q - 1}$ for some $Y \in \mathbbm{F}_{q^n}^{\ast}$ (see \cite{ik} for more
details). Hence, all projective norm graphs $\mathcal{N}\mathcal{G}_{q, n} (\lambda)$ ($\lambda \in \mathbbm{F}_q^*$) are isomorphic.
It follows immediately from the definition that
$\mathcal{N}\mathcal{G}_{q, n} (1)$ is a regular graph of order $q^n$ and
valency $(q^n - 1) / (q - 1)$. The eigenvalues of this graph is not hard to
compute. We present here only a sketch of the proof,
which follows the presentation of \cite{ap}. Let $A$ be the adjacency matrix of
$\mathcal{N}\mathcal{G}_{q, n} (\lambda)$. The rows and columns of this matrix
are indexed by the ordered pairs of the set $\mathbbm{F}_{q^n}$. Let $\chi$ be
a character of the additive group of $\mathbbm{F}_{q^n}$. One can check (see
\cite{ap}) that $\chi$ is an eigenvector of $A^2$ with eigenvalue $| \sum_{N (c) = 1}
\chi (c) |^2$ and all eigenvalues of $A^2$ are of this form. The trivial
character is corespondent to the large eigenvalue $(q^t - 1) / (q - 1)$ of
$A$. The others can be estimated using Weil's bound on the character sum (see
\cite[Theorem 2E(i)]{schmidt}), and the fact that all solutions of $N (X) = 1$ are all given
by $X = \sigma (Y) Y^{- 1} = Y^{q - 1}$ for some $Y \in
\mathbbm{F}_{q^n}^{\ast}$,
\[ \left| \sum_{N (c) = 1} \chi (c) \right| = \left| \frac{1}{q - 1} \sum_{d
   \in \mathbbm{F}_{q^n}^{\ast}} \chi (d^{q - 1}) \right| < \frac{(q - 2) q^{n
   / 2}}{q - 1} < q^{n / 2} . \]
Therefore, we have the following result.

\begin{lemma}\label{norm-graph-lemma}
  For any $n \geq 2$ and $\lambda \in \mathbbm{F}_q^{\ast}$, the projective
  norm graph $\mathcal{N}\mathcal{G}_{q, n} (\lambda)$ is a $(q^n, \frac{q^n -
  1}{q - 1}, q^{n / 2})$-graph. 
\end{lemma}

\subsection{Proofs of Theorems \ref{norm-equation-sesb}, \ref{norm-pinned-sesb}, \ref{norm-fs-sesb} and \ref{norm-ps-sesb}}

We are now ready to prove results on norm equations in Section \ref{norm-result-sesb}. For
any two set $\mathcal{A}, \mathcal{B} \subseteq \mathbbm{F}_{q^n} $ and
$\lambda \in \mathbbm{F}_q^{\ast}$, let $e_{\lambda} (\mathcal{A},
\mathcal{B}) =\#\{(X, Y) \in \mathcal{A} \times \mathcal{B}: N (X + Y) =
\lambda\}$. It follows from Corollary \ref{edge-c-sesb} and Lemma \ref{norm-graph-lemma} that
\[ \left| e_{\lambda} (\mathcal{A}, \mathcal{B}) - \frac{(q^n - 1)
   |\mathcal{A}| |\mathcal{B}|}{(q - 1) q^n} \right| < q^{n / 2}
   \sqrt{|\mathcal{A}| |\mathcal{B}|} . \]
Hence, $e_{\lambda} (\mathcal{A}, \mathcal{B}) > 0$ if $|\mathcal{A}|
|\mathcal{B}| \geq q^{n + 2}$. This completes the proof of Theorem \ref{norm-equation-sesb}.

Next, we consider the graph $G$ on the vertex set $\mathbbm{F}_{q^n}$. Two distinct
vertices $X, Y$ are connected by an $\lambda$-colored edge if and only if $N
(X + Y) = \lambda$. Note that we only use $(q - 1)$ colors $\lambda \in
\mathbbm{F}_q^{\ast}$. From Lemma \ref{norm-graph-lemma}, the graph $G$ is a $(q^n, q^{n - 1},
q^{n / 2})$-colored graph with $(q - 1)$ colors. Theorem \ref{norm-fs-sesb} and Theorem \ref{norm-ps-sesb}
now follow immediately from Theorem \ref{a-c-sesb} and Theorem \ref{colored-subgraph-sesb}. Finally, it follows
from Lemma \ref{e-pinned-sesb} that
\[ \sum_{X \in \mathcal{A}} \sum_{\lambda \in \mathbbm{F}_q^{\ast}}
   \mathcal{I}_X^{\lambda} (\mathcal{B}) = \sum_{\lambda \in
   \mathbbm{F}_q^{\ast}} \sum_{X \in \mathcal{A}} \mathcal{I}_X^{\lambda}
   (\mathcal{B}) = (1 - o (1)) |\mathcal{A}|q, \]
which implies that there exits $\mathcal{A}^{'} \subset \mathcal{A}$, $\mathcal{A}^{'} \sim \mathcal{A}$ such that
\[ \sum_{\lambda \in \mathbbm{F}_q^{\ast}}
   \mathcal{I}_X^{\lambda} (\mathcal{B}) = (1 - o (1)) q,\] 
for any $X \in \mathcal{A}^{'}$. This completes the proof of Theorem \ref{norm-pinned-sesb}.

\section{Product graphs - Proofs of Theorems \ref{bilinear-equation-sesb}, \ref{bilinear-pinned-sesb}, and \ref{bilinear-system-sesb}} \label{section:bilinear}

For any non-degenerate bilinear from $B (\cdot, \cdot)$ on
$\mathbbm{F}_q^d$, and for any $\lambda \in \mathbbm{F}$, the product graph
$B_{q, d} (\lambda)$ is defined as follows. The vertex set of the product
graph $B_{q, d} (\lambda)$ is the set $V (B_{q, d} (\lambda)) =\mathbbm{F}^d
\backslash (0, \ldots, 0)$. Two vertices $\tmmathbf{a}$ and $\tmmathbf{b} \in
V (B_{q, d} (\lambda))$ are connected by an edge, $(\tmmathbf{a},
\tmmathbf{b}) \in E (B_{q, d} (\lambda))$, if and only if $B (\tmmathbf{a},
\tmmathbf{b}) = \lambda$. When $\lambda = 0$, the graph is just a blow-up of a
variant of Erd\H{o}s-R\'enyi graph. The eigenvalues of this graph are easy to compute
(for example, see \cite{ak}). We will now study the product graph when $\lambda \in
\mathbbm{F}^{\ast}$.

\begin{lemma}\label{product-graph-lemma}
  For any $d \geq 2$ and $\lambda \in \mathbbm{F}^{\ast}$, the product graph,
  $B_{q, d} (\lambda),$ is a $(q^d - 1, q^{d - 1}, \sqrt{2 q^{d - 1}})$-graph.
\end{lemma}

\begin{proof}
  It is easy to see that $B_{q, d} (\lambda)$ is a regular graph of order $q^d
  - 1$ and valency $q^{d - 1}$. We now compute the eigenvalues of this
  multigraph (i.e. graph with loops). For any $\tmmathbf{a} \neq \tmmathbf{b}
  \in \mathbbm{F}^d \backslash (0, \ldots, 0)$, the system
  \[ B (\tmmathbf{a}, \tmmathbf{x}) = \lambda, B (\tmmathbf{b}, \tmmathbf{y})
     = \lambda, \tmmathbf{x} \in \mathbbm{F}^d \backslash (0, \ldots, 0), \]
  has $q^{d - 2}$ solutions when $\tmmathbf{a} \neq \alpha \tmmathbf{b}$ for
  all $\alpha \in \mathbbm{F}_q$, and no solution otherwise. Hence, for any two
  vertices $\tmmathbf{a} \neq \tmmathbf{b}$, $\tmmathbf{a}$ and $\tmmathbf{b}$
  have $q^{d - 2}$ common neighbors if $\tmmathbf{a}$ and $\tmmathbf{b}$ are
  linear independent, and no common neighbor otherwise. Let $A$ be the
  adjacency matrix of $\mathcal{P}_{q, d} (\lambda)$. It follows that
  \begin{equation}\label{e19-sesb} A^2 = q^{d - 2} J + (q^{d - 1} - q^{d - 2}) J - E, \end{equation}
  where $J$ is the all-one matrix, $I$ is the identity matrix, and $E$ is the
  adjacency matrix of the graph $\mathcal{B}_E$, where for any two vertices
  $\tmmathbf{a}$ and $\tmmathbf{b} \in V (B_{q, d} (\lambda))$,
  $(\tmmathbf{a}, \tmmathbf{b})$ is an edge of $\mathcal{B}_E$ \ if and only
  if $\tmmathbf{a}$ and $\tmmathbf{b}$ are linearly dependent. Therefore,
  $\mathcal{B}_E$ is a $(q - 1)$-regular graph, and all eigenvalues of $E$ are
  at most $q - 1$. Since $B_{q, d} (\lambda)$ is a $(q^d - 1)$-regular graph,
  $q^{d - 1}$ is an eigenvalue of $A$ with the all-one eigenvalue
  $\tmmathbf{1}$. The graph $B_{q, d} (\lambda)$ is connected so the
  eigenvalue $q^{d - 1}$ has multiplicity one. For any $\tmmathbf{a}$ and $\tmmathbf{b}$ with $B (\tmmathbf{a},
  \tmmathbf{b}) = \lambda$,  $\tmmathbf{a}$ and $\tmmathbf{b}$ are linearly
  independent and $\tmmathbf{a}$ and $\tmmathbf{b}$ have $q^{d - 2}$ common
  neighbors). Hence, the graph is not bipartite. Therefore, for any other eigenvalue $\theta$, $| \theta | < q^{d - 1}$.
  Let $\tmmathbf{v}_{\theta}$ denote the corresponding eigenvector
  of $\theta$. Note that $\tmmathbf{v}_{\theta} \in \tmmathbf{1}^{\bot}$, so
  $J\tmmathbf{v}_{\theta} = 0$. It follows from (\ref{e19-sesb}) that
  \[ (\theta^2 - q^{d - 1} + q^{d - 2})\tmmathbf{v}_{\theta} =
     E\tmmathbf{v}_{\theta} . \]
  Hence, $\tmmathbf{v}_{\theta}$ is also an eigenvector of $E$. Since all
  eigenvalues of $E$ is bounded by $q - 1$, we have
  \[ \theta^2 \leq q^{d - 1} - q^{d - 2} + q - 1 < q^d + q - 2 < 2 q^{d - 1} .
  \]
  (Note that, one can get $\theta^2 < q^{d - 1}$ when $d \geq 2$.) The lemma
  follows.
\end{proof}

Similarly as in the previous section, Theorems \ref{bilinear-equation-sesb}, \ref{bilinear-pinned-sesb}, and \ref{bilinear-system-sesb} follow immediately from Corollary \ref{edge-c-sesb}, Lemma \ref{e-pinned-sesb}, Theorem \ref{colored-subgraph-sesb}, and Lemma \ref{product-graph-lemma}.

\section{Sum-product graphs}

For any non-degenerate bilinear form $B (\cdot, \cdot)$ on
$\mathbbm{F}_q^d$ and for any $\lambda \in \mathbbm{F}_q$, the sum-product
graph $S B_{q, d} (\lambda)$ is defined as follows. The vertex set of the
sum-product graph $S B_{q, d} (\lambda)$ is the set $V (S B_{q, d})
=\mathbbm{F} \times \mathbbm{F}^d$. Two vertices $U = (a, \tmmathbf{b})$ and
$V = (c, \tmmathbf{d}) \in V (S B_{q, d})$ are connected by an edge, $(U, V)
\in E (S B_{q, d})$, if and only if $a + c + \lambda = B (\tmmathbf{b},
\tmmathbf{d})$. Our construction is similar to that of Solymosi in \cite{solymosi}.

\begin{lemma}\label{sp-graph-lemma}
  For any $d \geq 1$ and $\lambda \in \mathbbm{F}$, the sum-product graph, $S
  B_{q, d} (\lambda),$ is a $(q^{d + 1}, q^d, \sqrt{2 q^d})$-graph.
\end{lemma}

\begin{proof}
  It is easy to see that $S B_{q, d} (\lambda)$ is a regular graph of order
  $q^{d + 1}$ and valency $q^d$. We now compute the eigenvalues of this
  multigraph. For any $a, c \in \mathbbm{F}$ and
  $\tmmathbf{b} \neq \tmmathbf{d} \in \mathbbm{F}^d$, the system
  \[ a + u + \lambda = B (\tmmathbf{b}, \tmmathbf{v}), c + u + \lambda = B
     (\tmmathbf{d}, \tmmathbf{v}), \; u \in \mathbbm{F}, \tmmathbf{v} \in
     \mathbbm{F}^d \]
  has $q^{d - 1}$ solutions. (We can argue as follows. There are $q^{d - 1}$
  possibilities of $\tmmathbf{v}$ such that $B (\tmmathbf{b}-\tmmathbf{d},
  \tmmathbf{v}) = a - c$. For each choice of $\tmmathbf{v}$, there exists a
  unique $u$ satisfying the system.) If $\tmmathbf{b}=\tmmathbf{d}$ and $a
  \neq c$ then the system has no solution. Hence, for any two vertices $U =
  (a, \tmmathbf{b})$ and $V = (c, \tmmathbf{d}) \in V (S B_{q, d} (\lambda))$,
  if $\tmmathbf{b} \neq \tmmathbf{d}$ then $U$ and $V$ have exactly $q^{d -
  1}$ common neighbors, and if $\tmmathbf{b}=\tmmathbf{d}$ and $a \neq c$ then
  $U$ and $V$ have no common neighbor. Let $A$ be the adjacency matrix of $S
  B_{q, d} (\lambda)$. It follows that
  \begin{equation}\label{e20-sesb} A^2 = A A^T = q^{d - 1} J + (q^d - q^{d - 1}) I - E, \end{equation}
  where $J$ is the all-one matrix, $I$ is the identity matrix, and $E$ is the
  adjacency matrix of the graph $B_E$, where for any two vertices $U = (a,
  \tmmathbf{b})$ and $V = (c, \tmmathbf{d}) \in V (S B_{q, d} (\lambda))$,
  $(U, V)$ is an edge in $B_E$ if and only if $a \neq c$ and
  $\tmmathbf{b}=\tmmathbf{d}$. Since $S B_{q, d} (\lambda)$ is a $q^d$-regular
  graph, $q^d$ is an eigenvalue of $A$ with the all-one eigenvalue
  $\tmmathbf{1}$. The graph $S B_{q, d} (\lambda)$ is connected so the
  eigenvalue $q^d$ has multiplicity one. Besides, choose $\tmmathbf{b}, \tmmathbf{d} \in \mathbbm{F}_q^d$ such that $B
  (\tmmathbf{b}, \tmmathbf{d}) = 2 a \neq 0$, then $S B_{q, d} (\lambda)$
  contains the triangle with three vertices $(- a, \tmmathbf{0})$, $(a,
  \tmmathbf{b})$, and $(a, \tmmathbf{d})$. So the graph is not bipartite. Hence, for any other eigenvalue
  $\theta$, $| \theta | < q^d$. Let $\tmmathbf{v}_{\theta}$ denote
  the corresponding eigenvector of $\theta$. Note that $\tmmathbf{v}_{\theta}
  \in \tmmathbf{1}^{\bot}$, so $J\tmmathbf{v}_{\theta} = 0$. It follows from
  (\ref{e20-sesb}) that
  \[ (\theta^2 - q^d + q^{d - 1})\tmmathbf{v}_{\theta} =
     E\tmmathbf{v}_{\theta} . \]
  Hence, $\tmmathbf{v}_{\theta}$ is also an eigenvector of $E$. Since $B_E$ is
  a regular graph of order $q - 1$, the absolute value of any eigenvalue of
  $E$ is at most $q - 1$. This implies that
  \[ \theta^2 < q^d - q^{d - 1} + (q - 1) < 2 q^d . \]
  (Note that, one can get $\theta^2 < q^d$ when $d \geq 2$.) The lemma
  follows.
\end{proof}

\subsection{Proofs of Theorems \ref{sp-equation-sesb}, \ref{spe-sesb} and \ref{sp-system-sesb}}

Similarly, Theorem \ref{sp-equation-sesb} follows from Lemma \ref{sp-graph-lemma} and Corollary \ref{edge-c-sesb}; Theorem \ref{sp-system-sesb} follows
from Theorem \ref{a-c-sesb}, Theorem \ref{colored-subgraph-sesb} and Lemma \ref{sp-graph-lemma}. The proof of Theorem \ref{spe-sesb} remains. Without loss of generality, we can suppose that $0 \notin \mathcal{A}$. For any $\mathcal{A}
\subseteq \mathbbm{F}_q^{\ast}$, let $\mathcal{A}^{- 1} =\{1 / a : a \in
\mathcal{A}\} \subset \mathbbm{F}_q$, $\mathcal{E}_{\mathcal{A}} =
k\mathcal{A} \times (\mathcal{A} \cdot \mathcal{A})^{d - 1} \subset
\mathbbm{F}_q \times \mathbbm{F}_q^{d - 1}$, and $\mathcal{F}_{\mathcal{A}} =
(-\mathcal{A}) \times (1 /\mathcal{A})^{d - 1} \subset \mathbbm{F}_q \times
\mathbbm{F}_q^{d - 1}$.
Let $e (\mathcal{E}_{\mathcal{A}}, \mathcal{F}_{\mathcal{A}})$ be the number
of edges of $\mathcal{G}_{q, d}$ between $\mathcal{E}_{\mathcal{A}}$ and
$\mathcal{F}_{\mathcal{A}}$. It follows from Theorem \ref{spe-sesb} and Lemma \ref{sp-graph-lemma} that
\begin{equation} \label{e2-spb} e (\mathcal{E}_{\mathcal{A}}, \mathcal{F}_{\mathcal{A}}) \leq
   \frac{|\mathcal{E}_{\mathcal{A}} | |\mathcal{F}_{\mathcal{A}} |}{q} +
   \sqrt{2 q^{d - 1} |\mathcal{E}_{\mathcal{A}} | |\mathcal{F}_{\mathcal{A}}
   |} . \end{equation}
There is an edge between any two vertices $(a_1 + \ldots + a_d, a_1 b_1,
\ldots, a_{d - 1} b_{d - 1}) \in \mathcal{E}_{\mathcal{A}}$ and $(- a_d,
b_1^{- 1}, \ldots, b_{d - 1}^{- 1}) \in \mathcal{F}_{\mathcal{A}}$. Hence,
\begin{equation}\label{e3-spb} e (\mathcal{E}_{\mathcal{A}}, \mathcal{F}_{\mathcal{A}}) \geq
   |\mathcal{A}|^{2 d - 1} . \end{equation}
Putting (\ref{e2-spb}) and (\ref{e3-spb}) together, we have
\[ |\mathcal{A}|^{2 d - 1} \leq \frac{|\mathcal{A}|^d |\mathcal{A} \cdot
   \mathcal{A}|^{d - 1} |d\mathcal{A}|}{q} + \sqrt{q^d |\mathcal{A}|^d
   |\mathcal{A} \cdot \mathcal{A}|^{d - 1} |d\mathcal{A}|} . \]
Let $x = |\mathcal{A} \cdot \mathcal{A}|^{(d - 1) / 2} |d\mathcal{A}|^{1 /
2}$, then
\[ \frac{|\mathcal{A}|^d}{q} x^2 + q^{d / 2} x - |\mathcal{A}|^{2 d - 1} \geq
   0. \]
Solving this inequality gives us the desired bound for $x$, concluding the
proof of Theorem~\ref{spe-sesb}.

\section{Finite Euclidean graphs - Proofs of Theorems \ref{distance-pinned-sesb} and \ref{simplex-system-sesb}}

Let $Q$ be a non-degenerate quadratic form on $\mathbbm{F}_q^d$. For any $\lambda \in \mathbbm{F}_q$, the finite Euclidean
graph $E_q (d, Q, \lambda)$ is defined as the graph with vertex set $\mathbbm{F}_q^d$ and the edge
set

\begin{equation}
  E =\{(\tmmathbf{x}, \tmmathbf{y}) \in \mathbbm{F}_q^d \times \mathbbm{F}_q^d \, |\, \tmmathbf{x} \neq \tmmathbf{y}, \, Q (\tmmathbf{x} - \tmmathbf{y}) = \lambda\}.
\end{equation}

Recall that an $(n,d,\lambda)$-graph is called a Ramanujan graph if $\lambda \leq 2\sqrt{d-1}$. We also call an $(n,d,\lambda)$-graph asymptotic Ramanujan graph if  $\lambda \leq (2+o(1))\sqrt{d}$ when $n, d, \lambda \rightarrow \infty$. The spectrum of the finite Euclidean graphs $E_q(d,Q,\lambda)$ when $Q(\tmmathbf{x}) = x_1^2 +\ldots + x_d^2$ was first investigated by Medrano et al. \cite{medrano}, who proved that $E_q(d,Q,\lambda)$ are asymptotically Ramanujan for any $\lambda \neq 0$. Bannai, Shimabukuro and Tanaka \cite{bst} extended this result to arbitrary non-degenerate quadratic form $Q$ using the character tables of  association schemes of affine type (\cite{kwok}). The following theorem summaries the results from \cite[Sections~2-6]{bst} and \cite[Section 3]{kwok}.

\begin{theorem} (\cite{bst,kwok}) \label{euclidean graphs}
  Let $Q$ be a non-degenerate quadratic form on $\mathbbm{F}_q^d$. For any $\lambda \in \mathbbm{F}_q^{*}$, the graph $E_q (d, Q, \lambda)$ is a $(q^d, (1+o(1))q^{d-1}, 2q^{(d-1)/2})$-graph.
\end{theorem}

We are now ready to prove Theorems \ref{distance-pinned-sesb} and \ref{simplex-system-sesb}. We consider the graph $G$ on the vertex set $\mathbbm{F}_q^d$.
Two distinct vertices $\tmmathbf{x}, \tmmathbf{y}$ are connected by an $\lambda$-colored edge if and only if $Q(\tmmathbf{x} - \tmmathbf{y}) = \lambda$ (note that we only use $(q-1)$ colors $\lambda \in \mathbbm{F}_q^*$). From Theorem \ref{euclidean graphs}, the graph $G$ is a $(q^d, (1+o(1))q^{d-1}, 2q^{(d-1)/2})$-colored graph with $(q-1)$ colors. Theorems \ref{distance-pinned-sesb} and \ref{simplex-system-sesb} follow immediately from Lemma \ref{e-pinned-sesb} and Theorem \ref{colored-subgraph-sesb}.

\section{Finite non-Euclidean graphs - Proof of Theorem \ref{simplex-sphere-sesb}}

Let $Q$ be a non-degenerate quadratic form on $\mathbbm{F}_q^d$. For each element $\tmmathbf{x} \in S^d(Q)$, we denote the
pair of antipodes on $S^d(Q)$ containing $\tmmathbf{x}$ by $[\tmmathbf{x}]$. Let
$\Omega$ be the set of 
pairs of antipodes on the unit $Q$-sphere (or equivalently, the lines  through them). For a fixed $\gamma \in \mathbbm{F}_q$, the finite
non-Euclidean graph $\mathcal{P}_{q, d} (\gamma)$ has the vertex set $\Omega$
and the edge set
\[ \{([\tmmathbf{x}], [\tmmathbf{y}]) \in \Omega \times \Omega \, | \, \tmmathbf{x} \neq \tmmathbf{y}, \, Q(\tmmathbf{x} - \tmmathbf{y})= 2 \pm \gamma\}. \]
We will see that our graphs are the same as ones of Bannai, Hao, and Song in \cite{bhs}, and
of Bannai, Shimabukuro, and Tanaka in \cite{bst}.

Note that $\Omega$ can also be viewed as the set of all square-type non-isotropic
one-dimensional subspaces of $\mathbbm{F}_q^d$ with respect to the quadratic form $Q$. The simple orthogonal group $O_d
(\mathbbm{F}_q)$ acts transitively on $\Omega$ and yields a symmetric
association scheme $\Psi (O_d (\mathbbm{F}_q), \Omega)$ of class $(q + 1) /
2$. We have two cases.

Case I. Suppose that $d = 2m+1$. The relations of $\Psi(O_{2m+1}(\mathbbm{F}_q),\Omega)$ are given by
\begin{eqnarray*}
 R_1 & = & \{([\tmmathbf{x}],[\tmmathbf{y}]) \in \Omega \times \Omega \mid  Q(\tmmathbf{x}+\tmmathbf{y})  = 0\},\\
 R_i & = & \{([\tmmathbf{x}],[\tmmathbf{y}]) \in \Omega \times \Omega \mid  Q(\tmmathbf{x}+\tmmathbf{y})  = 2 + 2 \nu^{- (i
  - 1)} \} \, (2 \leq i \leq (q - 1) / 2)\\
 R_{(q+1)/2} & = & \{([\tmmathbf{x}],[\tmmathbf{y}]) \in \Omega \times \Omega \mid  Q(\tmmathbf{x}+\tmmathbf{y}) = 2\},
\end{eqnarray*}
where $\nu$ is a generator of the field $\mathbbm{F}_q$ (see \cite[Section 4]{bhs}).

Case II. Suppose that $d = 2m$. The relations of $\Psi(O_{2m}(\mathbbm{F}_q),\Omega)$ are given by
\begin{eqnarray*}
 R_i & = & \{([\tmmathbf{x}],[\tmmathbf{y}]) \in \Omega \times \Omega \mid  Q(\tmmathbf{x}+\tmmathbf{y}) = 2 + 2^{-1} \nu^{i} \} \, (1 \leq i \leq (q - 1) / 2)\\
 R_{(q+1)/2} & = & \{([\tmmathbf{x}],[\tmmathbf{y}]) \in \Omega \times \Omega \mid  Q(\tmmathbf{x}+\tmmathbf{y}) = 2\},
\end{eqnarray*}
where $\nu$ is a generator of the field $\mathbbm{F}_q$ (see \cite[Section 6]{bhs} and \cite[Section 2]{bst}).

The graphs $(\Omega, R_i)$ are not Ramanujan in general. They, however, are asymptotic Ramanujan for large $q$. The following theorem can be
derived easily as in the proofs of \cite[Theorem 2.2]{bst} and \cite[Theorem 5.1]{bst} from the
character tables of the association scheme $\Psi (O_d (\mathbbm{F}_q),
\Omega)$ (\cite[Tables VI, VII]{bhs} and \cite[Theorem 6.3]{bhs}). Note that \cite[Theorem
2.2]{bst} requires an additional restriction that $q = p^r$, where the exponent $r$
is odd. This restriction assures that the graphs $(\Omega, R_i)$ $(2 \leq
i \leq (q + 1) / 2$) are Ramanujan if $q$ is sufficiently large. Since we
only need our graphs to be asymptotic Ramanujan, we can apply \cite[Eq. (3)]{bst} instead
of \cite[Lemma 2.1]{bst} in the proof of \cite[Theorem 2.2]{bst} to remove this restriction.

\begin{theorem}(\cite{bhs,bst})\label{mt-bst} The graphs $(\Omega,R_i)$ $(2 \leq i \leq (q-1)/2)$ are regular of order $q^{d-1}(1+o(1))/2$ and valency $q^{d-2}(1+o(1))$. Let $\lambda$ be any eigenvalue of the graph $(\Omega,R_i)$ with $\lambda \neq$ valency of the graph, then 
\[|\lambda| \leq (2+o(1))q^{(d-2)/2}.\]
\end{theorem}

Let $G$ be a graph with the vertex set $\Omega$, and the edge set is colored by $\{R_i\}_{2\leq i \leq (q-1)/2}$.  Theorem \ref{mt-bst} implies that $G$ is a $(q^{d-1}(1+o(1))/2, q^{d-2}(1+o(1)),(2+o(1))q^{(d-2)/2})$-colored graph with $(q-3)/2$ colors. Theorem \ref{simplex-sphere-sesb} now follows immediately from Theorem \ref{colored-subgraph-sesb}.


\begin{thebibliography}{1}
\providecommand{\url}[1]{\texttt{#1}}
\expandafter\ifx\csname urlstyle\endcsname\relax
  \providecommand{\doi}[1]{doi: #1}\else
  \providecommand{\doi}{doi: \begingroup \urlstyle{rm}\Url}\fi

\bibitem{ak} N. Alon and M. Krivelevich, Constructive bounds for a Ramsey-type problem, \textit{Graphs
and Combinatorics} \textbf{13} (1997), 217-225.

\bibitem{ap}
N. Alon and P. Pudl\'ak,  Constructive lower bounds for off-diagonal Ramsey numbers, \textit{Israel J. Math.} \textbf{122} (2001), 243--251.

\bibitem{ars}
N. Alon, L. R\`onyai and T. Szab\`o, Norm-graphs: variations and applications, \textit{J. Combinatorial Theory, Ser. B} \textbf{76} (1999), 280--290.

\bibitem{as} N. Alon and J. H. Spencer,
\textit{The probabilistic method}, 2nd ed., Willey-Interscience, 2000.


\bibitem{bhs} E. Bannai, S. Hao and S.-Y. Song, Character tables of the association schemes of finite orthogonal groups acting on the nonisotropic points, \textit{Journal of Combinatorial Theory, Series A} \textbf{54} (1990), 164-200.

\bibitem{bst} E. Bannai, O. Shimabukuro and H. Tanaka, Finite analogues of non-Euclidean spaces and Ramanujan graphs, \textit{European Journal of Combinatorics} \textbf{25} (2004), 243--259.

\bibitem{b1} J. Bourgain, A Szemeredi type theorem for sets of positive density, \textit{Israel J. Math.} \textbf{54} (1986), no. 3, 307--331.

\bibitem{b2} J. Bourgain, Mordell's exponential sum estimate revisited, \textit{J. Amer. Math. Soc} \textbf{18} (2005), no. 2, 477-499.


\bibitem{bkt} J. Bourgain, N. Katz, and T. Tao, A sum product estimate in finite
fields and Applications, \textit{Geom. Funct. Analysis}, \textbf{14} (2004), 27–-57.

\bibitem{chapman} J. Chapman, M. B. Erdogan, Derrick Hart, Alex Iosevich, and Doowon Koh, Pinned distance sets, $k$-simplices, Wolff's exponent in finite fields and sum-product estimates, preprint (2009).

\bibitem{covert}
D. Covert, D. Hart, A. Iosevich and I. Uriarte-Tuero, An analog of the Furstenberg-Katznelson-Weiss theorem on triangles in sets of positive density in finite field geometries, preprint (2008). 

\bibitem{fkw} H. Furstenberg, Y. Katznelson, and B. Weiss, \textit{Ergodic theory and configurations in sets of positive density}, Mathematics of Ramsey theory, 184--198, Algorithms Combin., 5, Springer, Berlin (1990).


\bibitem{garaev2} M. Z. Garaev, The sum-product estimate for large subsets of prime fields,
\textit{Proc. Amer. Math. Soc.} \textbf{136} (2008), 2735-2739.

\bibitem{g} A. A. Glibichuk, Additive properties of product sets in an arbitrary finite fields, preprint. 

\bibitem{gk} A. A. Glibichuk and S. V. Konyagin, Additive properties of product sets
in fields of prime order, Centre de Recherches Math´ematiques, \textit{CRM
Proceedings and Lecture Notes}, \textbf{43}, 279–286 (2007).

\bibitem{gs} K. Gyarmati and A. S\'ark\"ozy, Equations in finite fields with restricted solution sets, II (algebraic equations), \textit{Acta Math. Hungar.} \textbf{119} (2008), 259--280.

\bibitem{hi} D. Hart and A. Iosevich, Ubiquity of simplices in subsets of vector spaces over finite fields, \textit{Analysis Mathematika}, \textbf{34} (2007).

\bibitem{ir} A. Iosevich and M. Rudnev, Erd\H{o}s distance problem in vector spaces
over finite fields, \textit{Trans. Amer. Math. Soc.}, \textbf{359} (2007), 6127-–6142.

\bibitem{ik} H. Iwaniec and E. Kowalski, \textit{Analytic number theory}, Amer. Math. Soc.,
Providence, RI, 2004.


\bibitem{ks} M. Krivelevich and B. Sudakov, Pseudo-random graphs, \textit{Conference on Finite and Infinite Sets Budapest}, Bolyai Society Mathematical Studies X, pp. 1--64.

\bibitem{kwok} W.M. Kwok, Character tables of association schemes of affine type, \textit{European J. Combin.} \textbf{13} (1992) 167--185.


\bibitem{magyar1} A. Magyar, On distance sets of large sets of integers points, \textit{Israel J. Math.} \textbf{164} (2008), 251--263.

\bibitem{magyar2} A. Magyar, $k$-point configurations in sets of positive density of $\mathbbm{Z}^n$, \textit{Duke Math J.} (to appear) (2007).

\bibitem
{medrano} A. Medrano, P. Myers, H. M. Stark and A. Terras, Finite analogues of Euclidean space,
\textit{Journal of Computational and Applied Mathematics}, \textbf{68} (1996), 221--238.

\bibitem{schmidt} W. G. Schmidt, \textit{Equations over Finite Fields: An Elementary Approach}, Springer LNM \textbf{536}, 1976.

\bibitem{shparlinski} I. Shparlinski, On the solvability of Bilinear Equations in Finite Fields, \textit{Glasg. Math. J.} \textbf{50} (2008), 523--529.

\bibitem{sarkozy} A. S\'ark\"ozy, On products and shifted products of residues modulo p, \textit{Integers: Electronic Journal of Combinatorial Number Theory}, \textbf{8}(2) (2008), A9.

\bibitem{solymosi} J. Solymosi, Incidences and the Spectra of Graphs, \textit{Building Bridges between Mathematics and Computer Science.} Vol. \textbf{19}. Ed. Martin Groetschel and Gyula Katona. Series: Bolyai Society Mathematical Studies. Springer, 2008. 499 -- 513.

\bibitem{vinh-erdos}
L. A. Vinh, Explicit Ramsey graphs and Erd\H{o}s distance problem over
finite Euclidean and non-Euclidean spaces, \textit{Electronic J. Combin.}, \textbf{15}
(2008), Article R5.

\bibitem{vinh-incidence}
L. A. Vinh, A Szemer\'edi-Trotter type theorem and sum-product estimate over finite fields, \textit{European Journal of Combinatorics} (to appear).

\bibitem{vinh-beq}
L. A. Vinh, On the sovability of systems of bilinear equations over finite fields, \textit{Proc. Amer. Math. Soc} (to appear).

\bibitem{vinh-tfs}
L. A. Vinh, Triangles in vector spaces over finite fields, \textit{Online Journal of Analytic Combinatorics} (to appear).

\bibitem{vinh-dps} 
L. A. Vinh, The Erd\H{o}s-Falconer distance problem on the unit sphere in vector spaces over finite fields, \textit{SIAM Journal on Discrete Mathematics}, accepted.

\bibitem{vinh-pg}
L. A. Vinh, On kaleidoscopic pseudo-randomness of finite Euclidean graphs, submitted.

\bibitem{van}
V. H. Vu, Sum-product estimates via directed expanders, \textit{Math. Res. Lett.} \textbf{15} (2008), no. 2, 375--388.

\end{thebibliography}
\end{document}